\documentclass[11pt,reqno]{amsart}
\usepackage{graphicx,amssymb,mathrsfs,amsmath,amsfonts,appendix}
\usepackage{subfigure}
\usepackage{amsthm}
\usepackage{paralist}
\usepackage{enumitem}
\usepackage[utf8]{inputenc}
\usepackage{comment}
\usepackage{amsmath}
\usepackage{mathtools}
\usepackage{esint}
\usepackage[colorlinks=true, pdfstartview=FitV, linkcolor=blue, citecolor=blue, urlcolor=blue]{hyperref}
\numberwithin{equation}{section}

\newtheorem{lemma}{Lemma}
\newtheorem{theorem}{Theorem}[section]

\usepackage[colorinlistoftodos]{todonotes}

\newcommand{\xd}{\textrm{d}}
\newcommand{\bB}[0]{B}

\newcommand{\ep}[0]{\varepsilon}
 
\newcommand{\embed}{\hookrightarrow}
\def\wstarto{\stackrel{*}{\rightharpoonup}}

\title[Nonlocal-to-local convergence of CH equations]{Nonlocal-to-local convergence of Cahn-Hilliard equations: 
Neumann boundary conditions\\
and viscosity terms}
\author[Elisa Davoli]{Elisa Davoli}
\address{Institut f\"ur Mathematik, University of Vienna, Oskar-Morgenstern-Platz 1, 1090 Vienna, Austria}
\email{elisa.davoli@univie.ac.at}
\author[Luca Scarpa]{Luca Scarpa}
\address{Institut f\"ur Mathematik, University of Vienna, Oskar-Morgenstern-Platz 1, 1090 Vienna, Austria}
\email{luca.scarpa@univie.ac.at}
\author[Lara Trussardi]{Lara Trussardi}
\address{Institut f\"ur Mathematik, University of Vienna, Oskar-Morgenstern-Platz 1, 1090 Vienna, Austria}
\email{lara.trussardi@univie.ac.at}

\begin{document}

\begin{abstract}
We consider
a class of
nonlocal viscous Cahn-Hilliard equations
with Neumann boundary conditions for the chemical potential.
The double-well potential 
is allowed to be singular 
(e.g.~of logarithmic type),
while the singularity of the
convolution kernel 
does not fall in
any available existence theory
under Neumann boundary 
conditions. 
We prove
well-posedness for the nonlocal equation
in a suitable variational sense.
Secondly, we show that 
the solutions to the 
nonlocal equation 
converge to the corresponding
solutions to the local equation, as the 
convolution kernels approximate a Dirac delta.
The asymptotic behaviour is analyzed 
by means of monotone analysis and 
Gamma convergence results,
both when the limiting local Cahn-Hilliard 
equation is of viscous type and of pure type.
\end{abstract}

\maketitle

\tableofcontents

%%%%%%%%%%%%%%%%%%%%%%%%%%%%%%%%%%%%%%%%%%%%%%%%%%%%%%%%%
\section{Introduction}
\label{sec:intro}
%%%%%%%%%%%%%%%%%%%%%%%%%%%%%%%%%%%%%%%%%%%%%%%%%%%%%%%%%

The aim of the present paper is to 
study the well-posedness and the asymptotic behaviour
as $\ep\searrow0$ of a family of nonlocal viscous Cahn-Hilliard equations with Neumann boundary conditions 
in the following form:
\begin{align}
    \label{eq1:NL}
    \partial_t u_\varepsilon - \Delta \mu_\varepsilon =0
    \qquad&\text{in } (0,T)\times\Omega\,,\\
    \label{eq2:NL}
    \mu_\varepsilon = \tau_\ep \partial_t u_\ep +(K_\varepsilon*1)u_\varepsilon -K_\varepsilon*u_\varepsilon +\Psi'(u_\varepsilon)
    -g_\ep \qquad&\text{in } (0,T)\times\Omega\,,\\
    \label{eq:2bis-Neumann}
    \partial_{\bf n} \mu_\ep=0\qquad&\text{on } (0,T)\times\partial\Omega\,,\\
    \label{eq3:NL}
    u_\ep(0)=u_{0,\ep} \qquad&\text{in } \Omega\,,
\end{align}
where $\Omega$ is a smooth bounded domain in $\mathbb{R}^d$
($d=2,3$), $T>0$ is a fixed final time, and
$\Psi'$ represents the derivative a 
double-well potential.
Moreover, $\ep>0$ is a fixed parameter,
$\tau_\ep>0$ is a
positive viscosity coefficient,
$K_\ep:\Omega\times\Omega\to\mathbb{R}$ 
is a suitable symmetric convolution kernel, and 
$g_\ep$ represents a distributed forcing term.
The variables $u_\ep$ and $\mu_\ep$ are
referred to as ``order parameter''
and ``chemical potential'', respectively.

The evolution problem \eqref{eq1:NL}--\eqref{eq3:NL}
is related to 
the gradient flow (in the $H^{-1}$-metric)
associated to a nonlocal free energy functional of the form
\begin{align} \label{eq:enfunGL}
\mathcal E_{\ep}(\varphi)&=\frac{1}{4}
\int_{\Omega}\int_{\Omega}K_\ep(x,y)
|\varphi(x)-\varphi(y)|^{2}\,\xd x\,\xd y
+ \int_{\Omega}\Psi(\varphi(x))\,\xd x\,.
\end{align}
Indeed, the contributions 
$(K_\ep*1) u_\ep - K_\ep*u_\ep + \Psi'(u_\ep)$ in the definition
of the chemical potential are obtained
exactly from the (sub)differentiation of the 
functional \eqref{eq:enfunGL}. The 
extra term $\tau_\ep\partial_t u_\ep$
represents on the other side a viscosity 
regularization, acting on the dissipation of the system.

The analysis of nonlocal models dates back to 
the early $90$'s, when
G.~Giacomin and J.~Lebowitz
investigated, in their seminal paper~\cite{GL},
a hydrodynamic limit of a microscopic model 
for a $d$-dimensional lattice gas
evolving via a Poisson nearest-neighbor process.
In that work, the authors derived 
a free energy functional in 
nonlocal form \eqref{eq:enfunGL},
and proposed the corresponding 
gradient flow to model phase-change in 
binary alloys. The viscous 
regularization in the definition of the chemical 
potential was originally introduced 
in the context of the local Cahn-Hilliard equation
by Novick-Cohen in \cite{novick-cohen}.
The mathematical literature on the nonlocal 
Cahn-Hilliard equation is widely developed:
we can mention, among many others, 
the contributions
\cite{ab-bos-grass-NLCH,bat-han-NLCH,gal-gior-grass-NLCH,
gal-grass-NLCH,han-NLCH}
and the references therein.\\

The rapidly growing attention 
to the 
nonlocal Cahn-Hilliard equation 
is due on the one hand to
its microscopic justification,
and on the other hand to 
its connection with the corresponding local
model. Indeed,
at least in a formal way, 
the nonlocal dynamics
approach the local ones
when the family of interaction kernels 
$(K_\ep)_\ep$ concentrates around the origin.
The main issue we assess in this paper is
the asymptotic convergence
of solutions to the nonlocal system
\eqref{eq1:NL}--\eqref{eq3:NL}
to the corresponding local one,
as
the data $(g_\ep)_\ep$ approximate a new source $g$
and the coefficients $\tau_\ep$
converge to a certain new viscosity parameter $\tau$.
The local form of the limiting
Cahn-Hilliard equation
reads
\begin{align}
    \label{eq1:L}
    \partial_t u - \Delta \mu =0
    \qquad&\text{in } (0,T)\times\Omega\,,\\
    \label{eq2:L}
    \mu = \tau \partial_t u -\Delta u +\Psi'(u)-g
    \qquad&\text{in } (0,T)\times\Omega\,,\\
    \label{eq:2bis-Neumann-L}
    \partial_{\bf n} u=0 \quad\text{and}\quad\partial_{\bf n} \mu=0\qquad&\text{on } (0,T)\times\partial\Omega\,,\\
    \label{eq3:L}
    u(0)=u_{0} \qquad&\text{in } \Omega\,,
\end{align}
where
$\tau\geq0$ is the limiting
viscosity parameter,
which is allowed to vanish.
The choices $\tau>0$ and $\tau=0$ correspond
to the viscous case and pure case, respectively.

As its nonlocal counterpart, the 
local Cahn-Hilliard equation
is related to the gradient flow in the 
$H^{-1}$ metric of the Ginzburg-Landau
free energy functional
\begin{equation}
\label{eq:enfunCH}
\mathcal E(\varphi)=
\frac12\int_{\Omega}|\nabla \varphi(x)|^2\,\xd x
+\int_\Omega \Psi(\varphi(x)) \,\xd x\,,
\end{equation}
in the sense that the contribution
$-\Delta u + \Psi'(u)$ results from the 
subdifferentiation of $\mathcal E$.
Again, the viscosity term
$\tau\partial_t u$ acts on the dissipation
of the system: if $\tau=0$, 
one recovers the so-called {\em pure}
Cahn-Hilliard equation, while if $\tau>0$
one obtains the {\em viscous} 
Cahn-Hilliard equation.
In our analysis, the nonlocal viscosity 
coefficients $(\tau_\ep)_\ep$ are assumed to 
be strictly positive, while the local 
coefficient $\tau$ is allowed to vanish.

The local Cahn-Hilliard equation
was first proposed in \cite{CH}
in relation to phase-change in 
metallic alloys and 
to spinodal decomposition (see \cite{maier-stan1}).
Nowadays, the model
is a widely used
in various contexts such as 
diffuse interface
modelling in physics and biology, 
with several applications to
tumor growth dynamics, 
image processing, and population dynamics.
From the mathematical point of view,
the local Cahn-Hilliard equation
has been studied thoroughly 
in the last decades, also 
in much more complex settings.
We mention, among many others,
the works
\cite{cher-gat-mir, cher-mir-zel, cher-pet, colli-fuk-CHmass,col-fuk-eqCH, col-gil-spr, 
gil-mir-sch, liu-wu-ARMA} on well-posedness
also under more general dynamic boundary conditions.
Some studies on nonlinear viscosity 
contributions have been proposed in 
\cite{bcst1, mir-sch, scar-VCHDBC}.
We also recall the contributions
\cite{col-far-hass-gil-spr, col-gil-spr-contr, col-gil-spr-contr2, CS17, hinter-weg}
dealing with optimal control problems,
as well as
\cite{col-fuk-diffusion, col-scar, gil-mir-sch-longtime} on asymptotics.
The local Cahn-Hilliard equation
has also been widely studied recently
in connection to 
diffuse-interface models for
fluid-dynamics: we refer to 
\cite{abels-ARMA, abels-gar-grun,
can-matt-nab-ARMA,
gal-grass-mir,
gal-grass-wu-ARMA} and the references 
therein.
\\

As already mentioned, 
the behaviour of the nonlocal Cahn-Hilliard 
equation ``approaches'' the one of the local 
equation when the family of convolution
kernels is sufficiently peaked around $0$.
The study of nonlocal-to-local convergence
of energy functionals in relation
to Sobolev spaces theory had been
carried out originally by
by J.~Bourgain, H.~Brezis, 
P.~Mironescu in \cite{BBM, BBM2}, 
and by V.~Mazy'a and T.~Shaposhnikova 
in \cite{MS, MS2}. 
This asymptotic analysis was also extended 
by A.~C.~Ponce in \cite{ponce04, ponce}, 
with studies on Gamma convergence and
nonlocal Poincar\'e-type inequalities.
A first criterion for the convergence
of gradient flows from the Gamma-convergence
of the respective energies was 
given by E.~Sandier and 
S.~Serfaty in \cite{sand-serf}
in a abstract setting and for smooth energies,
with applications to Ginzburg-Landau functionals
(see also 
\cite{lin, sand-serf-2, serf} for further details
in this direction).\\

In particular, the above-mentioned 
results \cite{ponce04, ponce} 
provide the pointwise
convergence
\[
  \lim_{\ep\searrow0}
  \mathcal E_\ep(\varphi)= \mathcal E(\varphi)
  \qquad\forall\,\varphi\in H^1(\Omega)
\]
as soon as the convolution kernels $(K_\ep)_\ep$
are chosen as
\begin{equation}\label{kernel}
K_\ep:\Omega\times\Omega\to[0,+\infty)\,,
\qquad
K_\ep(x,y):=\frac{\rho_\ep(|x-y|)}{|x-y|^2}\,,
\quad x,y\in\Omega\,,
\end{equation}
where $(\rho_\ep)_\ep$ is 
a suitable family of mollifiers converging to
a Dirac delta.

Building upon these variational convergences, 
in a previous contribution of ours
\cite{DRST} we rigorously derived
some nonlocal-to-local asymptotics 
of solutions to Cahn-Hilliard equations
in the setting of periodic boundary conditions
and with no viscosity effects.
The 
periodic setting adopted in \cite{DRST}
was fundamental to overcome the 
singular behaviour of the convolution kernel
\eqref{kernel}. Indeed, 
kernels in the form \eqref{kernel}
do not possess any $W^{1,1}$ regularity
(see for
example \cite[Remark 1]{col-gil-spr-OCNLphase}),
which is the usual minimum
requirement in the whole literature
on nonlocal Cahn-Hilliard systems. This 
resulted in the impossibility of framing 
the nonlocal problem in any available 
existence theory, and required an ad-hoc analysis.
In this direction, the 
arguments strongly relied on the
assumption of periodic boundary conditions.

The results in \cite{DRST} 
(see also \cite{MRT18} for a simpler case)
are very satisfactory since they
provide a novel contribution in 
the direction of local asymptotics
of Cahn-Hilliard equations.
Nevertheless,
the most natural choice of 
boundary conditions
in phase-field modelling if of 
no-flux type.
Consequently, it is 
crucial in this direction
to generalize 
the periodic framework to
other settings more
suited for applications.
The nonlocal-to-local convergence
of pure Cahn-Hilliard equations
with Neumann boundary conditions
was, to the authors' knowledge, 
still an open problem.
The main novelty
of the present paper is to finally extend
some rigorous 
nonlocal-to-local convergence
results for Cahn-Hilliard equations
to the case of
homogeneous Neumann boundary conditions.\\

Let us briefly describe now the main
difficulties arising in the 
case of Neumann boundary conditions.

The first hurdle has been already 
anticipated and concerns the 
regularity of the convolution kernel.
Indeed, in the form \eqref{kernel}
the kernel $K_\ep$ is not $W^{1,1}$,
and not even $L^1$ in dimension $d=2$.
This results in the necessity 
of rigorously formulate the nonlocal 
problem without relying on 
any available existence theory.
The main idea here is that 
even if the convolution operator 
$\varphi\mapsto K_\ep*\varphi$ may be
ill-defined under \eqref{kernel},
the nonlocal operator 
$B_\ep:\varphi\mapsto (K_\ep*1)\varphi
-(K_\ep*\varphi)$ appearing in
the equation \eqref{eq2:NL} can be 
rigorously defined instead.

The second main problem consists in 
the (im)possibility of
proving space regularity for
the solutions to the nonlocal equation
(i.e.~when $\ep>0$ is fixed).
If the convolution kernel is
$W^{1,1}$ this follows directly 
from the properties of the
convolution, i.e.~formally 
shifting the gradient operator
on the kernel as $\nabla(K_\ep*u_\ep)
=(\nabla K_\ep) * u_\ep$.
However,
for singular kernels as in \eqref{kernel}
this procedure fails.
Under periodic boundary conditions
(i.e.~working on the $d$-dimensional flat torus) 
the main idea to overcome this problem
was to use a certain integration-by-parts
formula, which hinges in turn on 
some compatibility conditions between
the convolution operator and the Laplace 
operator. More specifically, 
in \cite{DRST} the periodic setting 
allowed to prove a (formal) relation
in the form $\nabla (K_\ep*u_\ep)=
K_\ep*\nabla u_\ep$, from which 
one could deduce
$H^1$-regularity of the nonlocal solutions.
Nevertheless, under
Neumann boundary conditions
(i.e.~working on a bounded domain $\Omega\subset
\mathbb{R}^d$), in order to prove 
an analogous compatibility relation 
one is forced to extend the nonlocal solution
$u_\ep$ to $0$ outside $\Omega$. Clearly,
$H^1$-regularity in $\Omega$ does not 
imply $H^1$-regularity on the whole $\mathbb{R}^d$
for such extension.
This gives rise
to several extra boundary contribution terms
which blow up as the approximating parameter
vanishes.

The main consequence is that in the 
case of Neumann boundary conditions
one loses any $H^1$-estimate on the nonlocal
solutions.
It follows that the natural variational 
setting to frame the nonlocal problem
\eqref{eq1:NL}--\eqref{eq3:NL}
is not the usual one given by 
the triple $(H^1(\Omega), L^2(\Omega), 
H^1(\Omega)^*)$,
but instead 
an abstract one $(V_\ep, L^2(\Omega), V_\ep^*)$,
depending on $\ep$, where
$V_\ep$ represents, roughly speaking,
the domain of the nonlocal energy contribution
in \eqref{eq:enfunGL}.
As the inclusion $V_\ep\hookrightarrow L^2(\Omega)$ is not compact, 
one loses any reasonable 
compactness property on the approximated solutions
in order to pass to the limit in the 
nonlinearity.
This issue is overcome by the introduction
of the viscosity term
$\tau_\ep\partial_t
u_\ep$. Indeed, if $\tau_\ep$ is strictly 
positive one can 
show ``by hand'' a strong convergence
in $L^2(\Omega)$ for some regularized solutions,
even without relying on any $H^1$ estimates.

The third main problem concerns the boundary 
conditions of Neumann type for $u$ in the limiting
local problem.
Indeed, while the nonlocal system
is of order $2$ in space, hence it only 
needs one boundary condition (for the chemical 
potential), the limiting local equation 
is of order $4$ in space and requires
two boundary conditions instead: 
one for $\mu$ and one for $u$.
One of the major point is to understand
which is the natural extra boundary condition
for $u$, and how this one emerges when 
$\ep\searrow0$.
It is clear that the Neumann boundary condition
for the chemical potential is preserved 
by the local asymptotics. On the other hand,
the scenario for $u$ is more subtle:
the answer is implicitly given by 
studying the Gamma convergence of the nonlocal energies.
Indeed, in \cite{ponce}
Ponce proved a Gamma convergence result in
the form
\[
  \lim_{\ep\searrow0}\frac{1}{4}
  \int_{\Omega}\int_{\Omega}K_\ep(x,y)
  |\varphi_\ep(x)-\varphi_\ep(y)|^{2}\,\xd x\,\xd y
  =
  \begin{cases}
  \frac12\int_\Omega|\nabla\varphi(x)|^2\,\xd x
  \quad&\text{if } \nabla\varphi\in L^2(\Omega)\,,\\
  +\infty \quad&\text{otherwise}\,,
  \end{cases}
\]
whenever $\varphi_\ep \to \varphi$ in $L^2(\Omega)$.
Note that the limiting energy contribution
on the right-hand side is 
the potential associated to
the negative Laplacian with 
homogeneous Neumann boundary conditions.
Hence, this implicitly reveals that the ``correct''
choice of
boundary condition arising for $u$ in the local
limit is of Neumann type.
Such idea is indeed proved rigorously
performing the local asymptotics on the 
variational formulation for the 
nonlocal problem \eqref{eq1:NL}--\eqref{eq3:NL}.
The advantage of working using a 
variational approach is that the boundary 
conditions are implicitly contained in
the variational formulation itself,
and they have not to be tracked 
explicitly performing a pointwise analysis on 
the boundary.
\\

We are now in a position to
present the two main theorems 
that we prove in this paper.

The first main result is the well-posedness for the 
nonlocal system \eqref{eq1:NL}--\eqref{eq3:NL}
with Neumann boundary conditions 
when $\ep>0$ is fixed.
Here, the viscosity coefficient
$\tau_\ep$ is assumed to be strictly positive,
the convolution kernel is of the form \eqref{kernel},
and the double-well potential may be 
singular. In particular, we include in our
analysis all the typical examples of 
polynomial, logarithmic, and double-obstacle
potentials:
\begin{align*}
    &\Psi_{pol}(r):=\frac14(r^2-1)^2\,, \qquad
    r\in\mathbb{R}\,,\\
    &\Psi_{log}(r):=
    \frac{\vartheta}{2}\left[
    (1+r)\ln(1+r) + (1-r)\ln(1-r)
    \right] - \frac{\vartheta_0}{2}\,,
    \qquad r\in(-1,1)\,, \quad
    0<\vartheta<\vartheta_0\,,\\
    &\Psi_{doub}(r):=
    \begin{cases}
    c(1-r^2) \quad&\text{if } r\in[-1,1]\,,\\
    +\infty \quad&\text{otherwise}\,,
    \end{cases}
    \qquad c>0\,.
\end{align*}
In view of this, 
the derivative of $\Psi$ is interpreted
as a subdifferential in the sense of convex analysis,
and equation \eqref{eq2:NL} 
becomes a differential inclusion.
The proof of well-posedness is based
on a suitable approximation of the problem, 
given by a Yosida-type regularization on the 
nonlinearity and an additional elliptic 
local regularization in the chemical potential.
A novel abstract variational setting
$(V_\ep, L^2(\Omega), V_\ep^*)$ is introduced 
and uniform estimates on the 
approximated solutions are obtained.
Using the viscous contribution in the chemical 
potential, strong compactness in $L^2$
is recovered even with no
$H^1$-estimates on the solutions.
Strong convergences are then
proved and a passage to the limit 
provides solutions to the original 
nonlocal problem.

The second main result of this paper 
is the asymptotic analysis of the 
nonlocal system as $\ep\searrow0$.
Here, we assume that the forcing terms
$(g_\ep)_\ep$ converge to a certain 
source $g$,
and that the viscosity coefficients satisfy
\[
  \lim_{\ep\searrow0}\tau_\ep=\tau\,.
\]
Here, the coefficient $\tau$ 
is allowed to be nonnegative:
when $\tau>0$ we obtain then nonlocal-to-local
convergence of viscous Cahn-Hilliard equations, 
while if $\tau=0$ we obtain 
the local asymptotics of nonlocal viscous 
Cahn-Hilliard equations with vanishing 
viscosities.
The proof is based on uniform estimates 
in $\ep$ on the nonlocal solutions.
Here, the strong compactness in $L^2$ is obtained
by proving an ad-hoc compactness inequality
involving the family on functional spaces
$(V_\ep)_{\ep>0}$.
The identification of the local limit 
$-\Delta u$ is obtained through 
the combination of monotone analysis techniques
and Gamma-convergence results for the 
nonlocal energy functional \eqref{eq:enfunGL}.\\

We conclude by highlighting some possible
applications of our results
to phase-field modelling.

The relevance of nonlocal-to-local
convergence of Cahn-Hilliard equations
with Neumann boundary conditions
is significant: among many others, 
we can mention here 
possible connections with
optimal control of 
tumor growth models.
In the recent years, 
phase-field models
have been widely used in tumor
growth dynamics, both in the 
local case 
(see \cite{eb-knopf, garcke-lam2, 
garcke-lam, garcke-et-al,
garcke-lam-roc, garcke}
and the references therein)
and in the nonlocal case 
(see \cite{frig-lam-roc}
and \cite{porta-grass, mel-roc}
for nonlocal Cahn-Hilliard equations
with reaction terms).
One of the main advantages of the nonlocal
setting is that regularity results
on the solutions are usually easier 
to obtain, not needing to 
rely on elliptic regularity properties.
As a consequence, 
the availability of rigorous 
nonlocal-to-local convergence results
would give the opportunity to
approximate solutions
to local phase-field systems with 
the solutions to the corresponding nonlocal ones,
which are indeed simpler to handle
on the mathematical side.
For example,
refined regularity on the solutions
are fundamental when dealing with 
optimal control problems, 
in order to write first-order conditions
for optimality. Hence, possible 
outcomes of nonlocal-to-local 
asymptotics concern 
refined analysis of optimal control 
of phase-field systems, in terms
of passing to the (local) limit
within first-order conditions for optimality
for the nonlocal system.\\

The paper is structured in the following way.
In Section~\ref{s:main} we state the assumptions, and
we introduce the abstract variational settings. Section~\ref{s:results}
is devoted to 
present the two main results.
Section~\ref{proof1} contains the proof
of well-posedness of the nonlocal 
system \eqref{eq1:NL}--\eqref{eq3:NL},
while Section~\ref{proof2}
focuses on the proof of nonlocal-to-local
asymptotics.

%%%%%%%%%%%%%%%%%%%%%%%%%%%%%%%%%%%%%%%%%%%%%%%%%%%%%%%%%%
\section{Mathematical setting}
\label{s:main}
%%%%%%%%%%%%%%%%%%%%%%%%%%%%%%%%%%%%%%%%%%%%%%%%%%%%%%

\subsection{Assumptions}
Throughout the paper, $\Omega$ is a smooth bounded
domain in $\mathbb{R}^d$, with $d=2,3$, and $T>0$ is 
a fixed final time. 
We will use the notation
$Q_t:=(0,t)\times\Omega$ for every $t\in(0,T]$,
and set $Q:=Q_T$, and $\Sigma:=(0,T)\times \partial \Omega$.
Moreover, $(\rho_\ep)_{\ep>0}$
is a family of mollifiers
with the following properties
(see \cite{ponce04, ponce}):
\begin{align*}
    &\rho_\ep:\mathbb{R}\to[0,+\infty)\,,
    \qquad\rho_\ep\in L^1_{loc}(\mathbb{R})\,, \qquad
    \rho_\ep(r)=\rho_\ep(-r) \quad\forall\,r\in\mathbb{R}\,,
    \qquad\forall\,\ep>0\,;\\
    &\int_0^{+\infty}
    \rho_\ep(r)r^{d-1}\,\xd r =\frac2{C_d} \quad\forall\,\ep>0\,;\\
    &\lim_{\ep\searrow0}\int_\delta^{+\infty}
    \rho_\ep(r)r^{d-1}\,\xd r=0 \quad\forall\,\delta>0\,,
\end{align*}
where $C_d:=\int_{S^{d-1}}
|e_1\cdot\sigma|^2\,\xd\mathcal H^{d-1}(\sigma)$.
We define the family of 
convolution kernels as
\begin{equation}
    K_\ep:\Omega\times\Omega\to [0,+\infty)\,, \qquad
    K_\ep(x,y):=\frac{\rho_\ep(|x-y|)}{|x-y|^2}\,,
    \quad\text{for a.e.~}x,y\in\Omega\,, \qquad \ep>0\,.
\end{equation}

Throughout the paper, $\gamma:\mathbb{R}\to2^\mathbb{R}$ is a maximal
monotone graph with $0\in\gamma(0)$ and 
$\Pi:\mathbb{R}\to\mathbb{R}$ is 
$C_\Pi$-Lipschitz-continuous
with $\Pi(0)=0$.
It follows in particular that there exists
a proper, convex, lower semicontinuous function
$\hat\gamma:\mathbb{R}\to[0,+\infty]$ with $\hat\gamma(0)=0$
and $\partial\hat\gamma=\gamma$ in the sense of convex analysis. Similarly, we set
$\hat\Pi(s):=\int_0^s\Pi(r)\,\xd r$ for every $s\in \mathbb{R}$.
With these notations, the double-well potential 
$\Psi$ entering the system is represented by the sum 
$\hat\gamma + \hat\Pi$.

\subsection{Variational setting and preliminaries}
We introduce the functional spaces
\begin{equation*}
    H:=L^2(\Omega)\,, \qquad
    V:=H^1(\Omega)\,, \qquad
    W:=\left\{\varphi\in H^2(\Omega): 
    \partial_{\bf n}\varphi = 0 
    \text{ a.e.~on } \partial \Omega\right\}\,,
\end{equation*}
endowed with their natural norms, and 
we identify $H$ with its dual space in the usual way, 
so that
\[
  W \hookrightarrow V \hookrightarrow H
  \hookrightarrow V^* \hookrightarrow W^*
\]
where all the inclusions are continuous, dense, and compact. The Laplace operator with homogeneous 
Neumann conditions will be intended
both as a bounded linear operator
\[
  -\Delta:V\to V^*\,, \qquad
  \langle-\Delta\varphi,\zeta\rangle_V:=
  \int_\Omega\nabla\varphi(x)
  \cdot\nabla\zeta(x)\,\xd x\,, \quad\varphi,\zeta\in V\,,
\]
and as unbounded linear operator on $H$
with domain $W$.
For every $\varphi\in V^*$,
we use the notation 
$\varphi_\Omega:=\frac1{|\Omega|}
\langle \varphi,1\rangle_{V}$ for the mean
value on $\Omega$.
As a direct consequence
of the Poincar\'e-Wirtinger inequality it holds 
that
\[
  -\Delta:\{\varphi\in V:\varphi_\Omega=0\}\to
  \{\varphi\in V^*:\varphi_\Omega=0\}
\]
is a linear isomorphism. 
We will denote its inverse by
\[
\mathcal N:\{\varphi\in V^*:\varphi_\Omega=0\}\to
  \{\varphi\in V:\varphi_\Omega=0\}\,.
\]

For every $\ep>0$, we set
\begin{align*}
    V_\ep&:=
    \left\{\varphi\in L^2(\Omega):
    \int_{\Omega}\int_\Omega
    K_\ep(x,y)|\varphi(x)-\varphi(y)|^2\,\xd x\,\xd y
    <+\infty\right\}\,,\\
    E_\ep(\varphi)&:=
    \frac14\int_{\Omega}\int_\Omega
    K_\ep(x,y)|\varphi(x)-\varphi(y)|^2\,\xd x\,\xd y\,,
    \quad\varphi\in V_\ep\,,
\end{align*}
and 
\begin{align*}
    W_\ep&:=
    \left\{\varphi\in L^2(\Omega):
    x\mapsto\int_{\Omega}
    K_\ep(x,y)(\varphi(x)-\varphi(y))\,\xd y
    \in L^2(\Omega)\right\}\,,\\
    B_\ep(\varphi)(x)&:=
    \int_{\Omega}
    K_\ep(x,y)(\varphi(x)-\varphi(y))\,\xd y\,,
    \quad\text{for a.e.~}x\in\Omega\,,\quad
    \varphi\in W_\ep\,.
\end{align*}
We point out that 
$E_\ep:V_\ep\to[0,+\infty)$ is convex 
and $B_\ep: H\to H$ is a linear 
unbounded operator with domain $W_\ep$.
Additionally, we define the maps
\[
  \|\cdot\|_{V_\ep}:V_\ep\to[0,+\infty)\,,
  \qquad
  \|\cdot\|_{W_\ep}:W_\ep\to[0,+\infty)
\]
as
\[
  \|\varphi\|_{V_\ep}:=\sqrt{\|\varphi\|_H^2+
  2E_\ep(\varphi)}\,, \qquad
  \|\varphi\|_{W_\ep}:=\sqrt{\|\varphi\|_H^2
  +\|B_\ep(\varphi)\|_H^2}\,,
\]
and the bilinear forms
\[
  (\cdot,\cdot)_{V_\ep}:V_\ep\times V_\ep\to[0,+\infty)\,,
  \qquad
  (\cdot,\cdot)_{W_\ep}:W_\ep\times W_\ep\to[0,+\infty)
\]
as
\begin{align*}
  (\varphi_1,\varphi_2)_{V_\ep}&:=(\varphi_1,\varphi_2)_H+
  \frac12\int_\Omega\int_\Omega K_\ep(x,y)
  (\varphi_1(x)-\varphi_1(y))(\varphi_2(x)-\varphi_2(y))
  \,\xd x\,\xd y\,,\\
  (\varphi_1,\varphi_2)_{W_\ep}&:=(\varphi_1,\varphi_2)_H+
  (B_\ep(\varphi_1), B_\ep(\varphi_2))_H\,.
\end{align*}

We collect some properties in the next lemma.
\begin{lemma}
\label{lemma:properties}
  The following properties hold for every $\ep>0$.
  \begin{enumerate}
  \item The maps $\|\cdot\|_{V_\ep}$ and $\|\cdot\|_{W_\ep}$
  are complete norms on $V_\ep$ and $W_\ep$, respectively.
  \item The bilinear forms
  $(\cdot,\cdot)_{V_\ep}$ and $(\cdot,\cdot)_{W_\ep}$
  are scalar products on $V_\ep$ and $W_\ep$
  inducing the norms
  $\|\cdot\|_{V_\ep}$ and $\|\cdot\|_{W_\ep}$, respectively.
  In particular, $V_\ep$ and $W_\ep$ are Hilbert spaces.
  \item For every $\sigma\in(0,1]$ we have 
  $C^{0,\sigma}(\overline\Omega)\embed W_\ep$
  continuously, and there exists $C_{\ep,\sigma}>0$ 
  such that 
  \[
  B_\ep(\varphi)\in L^\infty(\Omega)\,, \quad
  \|B_\ep(\varphi)\|_{L^\infty(\Omega)}\leq
  C_{\ep,\sigma}\|\varphi\|_{C^{0,\sigma}(\overline\Omega)}
  \qquad\forall\,\varphi\in C^{0,\sigma}(\overline\Omega)\,.
  \]
  \item The following inclusions are continuous and dense:
  \[
  W_\ep \embed V_\ep
  \embed H\,.
  \]
  Moreover, $(B_\ep, W_\ep)$ is maximal monotone on $H$.  \item The unbounded linear operator $B_\ep:H\to H$
  extends to a bounded linear operator 
  $B_\ep:V_\ep\to V_\ep^*$, and it holds that
  \[
  \|B_\ep(\varphi)\|_{V_\ep^*}\leq\|\varphi\|_{V_\ep}
  \quad\forall\,\varphi\in V_\ep\,.
  \]
  \item The map $E_\ep:V_\ep\to[0,+\infty)$
  is of class $C^1$ and $DE_\ep=B_\ep:V_\ep\to V_\ep^*$.
  \end{enumerate}
\end{lemma}
\begin{proof}
  {\sc Step 1}:
  properties {\em (1)--(2)}.
  It is clear that $\|\cdot\|_{V_\ep}$ and $\|\cdot\|_{W_\ep}$
  are norms on $V_\ep$ and $W_\ep$, respectively.
  Let now $(y_n)_n$ be a Cauchy sequence in $V_\ep$: then in 
  particular it is a Cauchy sequence in $H$, so there exists $y\in H$ such that 
  $y_n\to y$ in $H$. By lower semicontinuity it follows that 
  $y\in V_\ep$ as well, and that $y_n\to y$ in $V_\ep$.
  A similar argument shows that $W_\ep$ is complete as well.
  A direct computation shows that
  $(\cdot,\cdot)_{V_\ep}$ and $(\cdot,\cdot)_{W_\ep}$
  are scalar products inducing the norms above.\\
  
  {\sc Step 2}:
  property {\em (3)}.
  For every $\varphi\in 
  C^{0,\sigma}(\overline\Omega)$, we have
  \[
  |B_\ep(\varphi(x))|\leq\int_\Omega
  \rho_\ep(|x-y|)
  \frac{|\varphi(x)-\varphi(y)|}{|x-y|^2}\,\xd y\leq
  \|\varphi\|_{C^{0,\sigma}(\overline\Omega)}
  \int_\Omega\frac{\rho_\ep(|x-y|)}{|x-y|^{2-\sigma}}\xd y,
  \]
  where
  \begin{align*}
  \int_\Omega\frac{\rho_\ep(|x-y|)}{|x-y|^{2-\sigma}}\xd y&=
  \int_{\Omega-x}\frac{\rho_\ep(|z|)}{|z|^{2-\sigma}}\xd z
  \leq\int_{\mathbb{R}^d}
  \frac{\rho_\ep(|z|)}{|z|^{2-\sigma}}\xd z=
  \int_{\{|z|\leq1\}}
  \frac{\rho_\ep(|z|)}{|z|^{2-\sigma}}\xd z
  +\int_{\{|z|>1\}}
  \frac{\rho_\ep(|z|)}{|z|^{2-\sigma}}\xd z\\
  &\leq\max_{|r|\leq1}\rho_\ep(r)\int_{\{|z|\leq1\}}
  \frac1{|z|^{2-\sigma}}\xd z +
  \int_{\{|z|>1\}}\rho_\ep(|z|)\,\xd z\,.
  \end{align*}
  The first term on the right-hand side is finite 
  since $2-\sigma<d$, while the second term can
  be written as
  \[
  |S^{d-1}|\int_1^{+\infty}\rho_\ep(r)r^{d-1}\,\xd r<+\infty
  \]
  by the assumptions on $(\rho_\ep)_\ep$. 
  The thesis follows by the arbitrariness of $x\in\Omega$.\\
  
  {\sc Step 3}:
  property {\em (4)}. First of all
  the fact that the inclusion
  $V_\ep\embed H$ is continuous is trivial
  by the definition of $\|\cdot\|_{V_\ep}$. Second, 
  for $\varphi\in W_\ep$, a direct computation shows that
  \[
  E_\ep(\varphi)=\frac14\int_\Omega\int_\Omega K_\ep(x,y)
  |\varphi(x)-\varphi(y)|^2\,\xd x\,\xd y=
  \frac12\int_\Omega B_\ep(\varphi(x))\varphi(x)\,\xd x
  \leq\frac12\|B_\ep(\varphi)\|_H\|\varphi\|_H\,,
  \]
  so that $W_\ep\embed V_\ep$ continuously. 
  The density of $V_\ep$ in $H$ follows from the density 
  of $C^{0,\sigma}(\overline\Omega)$ in $H$
  and the fact that $C^{0,\sigma}(\overline\Omega)
  \subset W_\ep\subset V_\ep$. 
  
  The monotonicity of $B_\ep$ is a direct consequence of its definition. We proceed by showing that it is maximal monotone.
  Let $\varphi\in H$. For every $\lambda,\delta>0$
  the elliptic problem
  \begin{equation}
      \label{eq:elliptic}
  \varphi_{\delta\lambda} + 
  \lambda\Delta^2\varphi_{\delta\lambda}
  +\delta B_\ep(\varphi_{\delta\lambda})=\varphi
  \end{equation}
  admits a unique solution 
  $\varphi_{\delta\lambda}\in W\embed 
  C^{0,1/4}(\overline\Omega)\embed W_\ep$.
  Fix $\delta>0$.
  Testing \eqref{eq:elliptic} by $\varphi_{\delta\lambda}$ and
  using the monotonicity of $B_\ep$ and
  the Young inequality,
  it follows that 
  \[
  \|\varphi_{\delta\lambda}\|_H^2
  +\lambda\|\Delta\varphi_{\delta\lambda}\|_H^2
  \leq \frac12\|\varphi\|_H^2+
  \frac12\|\varphi_{\delta\lambda}\|_H^2
  \quad\forall\,\lambda>0\,.
  \]
  Thus, by comparison there exists a positive 
  constant $M$ such that
  \[
  \|\varphi_{\delta\lambda}\|_H^2
  +\lambda\|\Delta\varphi_{\delta\lambda}\|_H^2
  +\|B_\ep(\varphi_{\delta\lambda})\|_{W^*}\leq M
  \quad\forall\,\lambda>0\,.
  \]
  We infer that there exist $\varphi_\delta\in H$ and $\eta_\delta\in W^*$ such that, as $\lambda\searrow0$,
  $\lambda\varphi_{\delta\lambda}\to 0$
  in $W$, 
  $\varphi_{\delta\lambda}\rightharpoonup\varphi_\delta$
  in $H$, and $B_\ep(\varphi_{\delta\lambda})
  \rightharpoonup \eta_\delta$ in $W^*$, from
  which $\varphi_\delta + \delta\eta_\delta =\varphi$.
  It follows by comparison that $\eta_\delta\in H$.
  For all $\zeta\in W$, 
  by the symmetry of $B_\ep$ there holds
  \[
  (\eta_\delta,\zeta)_H=\lim_{\lambda\to0}
  (B_\ep(\varphi_{\delta\lambda}),\zeta)_H=
  \lim_{\lambda\to0}
  (\varphi_{\delta\lambda}, B_\ep(\zeta))_H
  =(\varphi_\delta, B_\ep(\zeta))_H=
  (B_\ep(\varphi_\delta), \zeta)_H\,,
  \]
  so we conclude that $\varphi_\delta\in W_\ep$ and
  $\eta_\delta=B_\ep(\varphi_\delta)$.
  Hence,
  \begin{equation}
  \label{eq:phi-delta}
      \varphi_\delta + \delta B_\ep(\varphi_\delta) =
  \varphi\qquad\forall\,\delta>0\,.
  \end{equation}
  This proves that $B_\ep$ is a maximal monotone
  operator on $H$
  (see \cite[Thm.~2.2]{barbu-monot}).
  Testing now \eqref{eq:phi-delta} by $\varphi_\delta$ and using Young inequality
  it is immediate to see that
  \begin{equation}
  \label{eq:bd-n-H}
  \frac12\|\varphi_\delta\|_H^2 + 
  \delta(B_\ep\varphi_\delta, \varphi_\delta)_H\leq
  \frac12\|\varphi\|_H^2\,.
  \end{equation}
  If additionally $\varphi\in V_\ep$, 
  testing \eqref{eq:phi-delta} by $B_\ep(\varphi_\delta)$
  and using H\"older and Young inequalities
  yields
  \begin{align}
  \label{eq:bd-en}
  2 E_\ep(\varphi_\delta) +
  \delta\|B_\ep(\varphi_\delta)\|_H^2&=
  (B_\ep(\varphi_\delta), \varphi_\delta)_H +
  \delta\|B_\ep(\varphi_\delta)\|_H^2
  =(B_\ep(\varphi_\delta), \varphi)_H\\
  &\nonumber=\frac12\int_\Omega\int_\Omega
  K_\ep(x,y)(\varphi_\delta(x)-\varphi_\delta(y))
  (\varphi(x)-\varphi(y))\,\xd x\,\xd y\\
  &\nonumber\leq2
  \sqrt{E_\ep(\varphi)}\sqrt{E_\ep(\varphi_\delta)}\leq
  E_\ep(\varphi_\delta) +
  E_\ep(\varphi)\,.
  \end{align}
  We deduce that, as $\delta\searrow0$, 
  $\delta B_\ep(\varphi_\delta)\to 0$ in $H$. Hence, by \eqref{eq:phi-delta},
  $\varphi_\delta\to \varphi$ in $H$. By combining \eqref{eq:bd-n-H} and \eqref{eq:bd-en}, we obtain that 
  $\|\varphi_\delta\|_{V_\ep}\leq\|\varphi\|_{V_\ep}$
  for every $\delta >0$. As $V_\ep$ is uniformly convex,
  this implies that $\varphi_\delta\to \varphi$ in $V_\ep$,
  so that $W_\ep\embed V_\ep$ densely.\\
  
  {\sc Step 4}:
  property {\em (5)}.
  For every $\varphi\in W_\ep$ and $\zeta\in V_\ep$,
  by the H\"older inequality we have
  \[
  \left(B_\ep(\varphi),\zeta\right)_H
  =\frac12\int_\Omega\int_\Omega
  K_\ep(x,y)(\varphi(x)-\varphi(y))(\zeta(x)-\zeta(y))
  \,\xd x\,\xd y
  \leq 
  2\sqrt{E_\ep(\varphi)}\sqrt{E_\ep(\zeta)}\,.
  \]
  This implies that for every $\varphi\in W_\ep$, the 
  operator
  \[
  \zeta\mapsto(B_\ep(\varphi),\zeta)_H\,, 
  \quad \zeta\in V_\ep\,,
  \]
  is linear and continuous on $V_\ep$, and such that 
  \[
  \|\zeta\mapsto(B_\ep(\varphi),\zeta)_H\|_{V_\ep^*}\leq
  \|\varphi\|_{V_\ep} \qquad\forall\,\varphi\in W_\ep\,.
  \]
  Since $W_\ep\embed V_\ep$ is dense, we deduce that $B_\ep$
  extends to a bounded linear operator from $V_\ep$ to
  $V_\ep^*$, and the thesis follows.\\
  
  {\sc Step 5}:
  property {\em (6)}. We observe that $E_\ep:V_\ep\to[0,+\infty)$ is convex 
  and lower semicontinuous. A direct computation also 
  shows that $DE_\ep=B_\ep$ in the sense of G\^ateaux:
  since $B_\ep:V_\ep\to V_\ep^*$ is linear and continuous,
  the thesis follows.
\end{proof}

The next lemma shows some boundedness properties of the 
family $(B_\ep)_\ep$, uniformly in $\ep$.
\begin{lemma}
\label{lemma:other-prop}
  The following inclusion is continuous
  \[
  V\embed V_\ep\,,
  \]
  and there exists a constant $C$, independent of $\ep$, such that
  \[
  \|\varphi\|_{V_\ep}\leq C\|\varphi\|_{V} \quad\forall\,\varphi\in V\,.
  \]
  For every $\varphi,\zeta\in V$, there holds
  \begin{equation}
  \label{eq:limits}
  \lim_{\ep\searrow0}E_\ep(\varphi)=
  \frac12\int_\Omega|\nabla\varphi(x)|^2\xd x\,,\qquad
  \lim_{\ep\searrow0}\langle B_\ep(\varphi_1),
  \varphi_2\rangle_{V_\ep}=
  \int_\Omega\nabla\varphi_1(x)\cdot
  \nabla\varphi_2(x)\,\xd x\,.
  \end{equation}
  Finally, for every $\varphi\in H$ and for every 
  sequence $(\varphi_\ep)_{\ep>0}\subset H$ 
  with
  $\varphi_\ep\to\varphi$ in $H$, we have
  \[
    \liminf_{\ep\searrow0}E_\ep(\varphi_\ep)
    \geq
    E(\varphi):=
    \begin{cases}
  \frac12\int_\Omega|\nabla\varphi(x)|^2\,\xd x
  \quad&\text{if } \varphi\in V\,,\\
  +\infty &\text{if } \varphi\in H\setminus V\,.
  \end{cases}\,
  \]
  In other words, $(E_\ep)_{\ep>0}$
  $\Gamma$-converges to $E$ with respect 
  to the norm-topology of $H$.
\end{lemma}
\begin{proof}
  By \cite{ponce}, there is a constant $C>0$
  independent of $\ep$ such that 
  \[
  E_\ep(\varphi)\leq C\|\nabla\varphi\|_H^2
  \qquad\forall\,\varphi\in V\,,
  \]
  from which the first part of the thesis follows directly. The first
  limit in \eqref{eq:limits} is also a direct consequence of \cite{ponce},
  the second limit in \eqref{eq:limits} can be proved exactly in the
  same way as \cite[\S~1]{DRST}.
  Finally, by the $\Gamma$-convergence result in
  \cite[Thm.~8]{ponce04}, we know that 
  \[
  \liminf_{\ep\searrow0}E_\ep(\varphi_\ep)\geq
  \operatorname{sc--}\tilde E(\varphi)\,,
  \]
  where $\operatorname{sc--}\tilde E$
  is the lower semicontinuous envelope of
  \[
  \tilde E:H\to[0,+\infty]\,, \qquad
  \tilde E(\varphi):=
    \begin{cases}
  \frac12\int_\Omega|\nabla\varphi(x)|^2\,\xd x
  \quad&\text{if } \varphi\in C^1(\overline\Omega)\,,\\
  +\infty &\text{otherwise }\,,
  \end{cases}
  \]
  i.e.
  \[
  \operatorname{sc--}\tilde E(\varphi)=
  \inf\left\{\liminf_{n\to\infty}
  \tilde E(\zeta_n): \zeta_n\to\varphi \quad\text{in } H\right\}\,.
  \]
  It is a standard matter to check that 
  $\operatorname{sc--}\tilde E=E$, so that 
  the thesis follows.
\end{proof}

The last result of this section is 
a compactness criterion involving
the family of operators $(E_\ep)_\ep$.
The following lemma is fundamental as
we do not have any compactness properties 
for the inclusions 
of the spaces $V_\ep$ and $W_\ep$.
For the proof we refer to \cite[Lemma.~4]{DRST}.
\begin{lemma}
\label{lemma:delta-est}
For every $\delta>0$ there exist two constants $C_{\delta}>0$ and $\ep_\delta>0$ such that,
for every sequence 
$(\varphi_\ep)_{\ep\in(0,\ep_\delta)}\subset V_\ep$ there holds
\[
    \|\varphi_{\ep_1}-\varphi_{\ep_2}\|^2_{H}\leq 
    \delta \left(E_{\ep_1}(\varphi_{\ep_1})
    + E_{\ep_2}(\varphi_{\ep_2})\right)
    +C_{\delta}
    \|\varphi_{\ep_1}-\varphi_{\ep_2}\|^2_{V^*}
    \qquad\forall\,\ep_1,\ep_2\in(0,\ep_\delta)\,.
\]
\end{lemma}

%%%%%%%%%%%%%%%%%%%%%%%%%%%%%%%%%%%%%%%%%%%%%%%%
\section{Main results}
\label{s:results}
%%%%%%%%%%%%%%%%%%%%%%%%%%%%%%%%%%%%%%
Before stating our main results, we recall 
that the local Cahn-Hilliard equation
is well-posed in the following sense.

\begin{theorem}
  \label{th:wp_loc}
  Let $\tau\geq0$ and 
  \begin{align}
      \label{ip1_loc}
      &u_0 \in V\,, \qquad
      \hat\gamma(u_0)\in L^1(\Omega)\,, \qquad
      (u_0)_\Omega \in\operatorname{Int}D(\gamma)\,,\\
      \label{ip2_loc}
      &g\in L^2(0,T; H)\,, \qquad
      g\in H^1(0,T; H) \quad\text{if } \tau=0\,.
  \end{align}
  Then, there exists a triple 
  $(u,\mu,\xi)$ such that 
  \begin{align}
      \label{u_loc}
      &u \in H^1(0,T; V^*) \cap L^\infty (0,T; V)
      \cap L^2(0,T; W)\,,\qquad
      \tau u \in H^1(0,T; H)\,,\\
      \label{mu_loc}
      &\mu \in L^2(0,T; V)\,, \qquad
      \tau\mu\in L^2(0,T; W)\,,\\
      \label{xi_loc}
      &\xi \in L^2(0,T; H)\,, 
      \qquad \xi\in\gamma(u) \quad\text{a.e.~in } Q\,,\\
      \label{eq1_loc}
      &\partial_t u - \Delta\mu = 0 \quad\text{in } V^*\,,
      \quad\text{a.e.~in } (0,T)\,,\\
      \label{eq2_loc}
      &\mu=\tau\partial_t u - \Delta u + \xi + \Pi(u) - g
      \quad\text{a.e.~in } Q\,,\\
      \label{init_loc}
      &u(0)=u_0 \quad\text{a.e.~in } \Omega\,.
  \end{align}
  Moreover, the solution component
  $u$ is unique, and the solution 
  components $\mu$ and $\xi$
  are unique if $\gamma$ is single-valued.
\end{theorem}
\begin{proof}
We refer to \cite{col-gil-spr} for a proof in a 
more general setting.
\end{proof}

The first result of this paper is the well-posedness of
the nonlocal viscous Cahn-Hilliard equation 
complemented by Neumann boundary conditions
for the chemical potential. 

\begin{theorem}
 \label{th:wp_nloc}
  Let $\ep>0$ and $\tau_\ep>0$ be fixed.
  Then for every $(u_{0,\ep},g_\ep)$ with
  \begin{align}
      \label{ip1_nloc}
      &u_{0,\ep} \in V_\ep\,, \qquad
      \hat\gamma(u_{0,\ep})\in L^1(\Omega)\,, \qquad
      (u_{0,\ep})_\Omega \in
      \operatorname{Int}D(\gamma)\,,\\
      \label{ip2_nloc}
      &g_\ep \in L^2(0,T; H)\,,
  \end{align}
  there exists a triple
  $(u_\ep,\mu_\ep,\xi_\ep)$ such that 
  \begin{align}
  \label{u_nloc}
  &u_\varepsilon \in H^1(0,T; H) \cap
  L^\infty(0,T; V_\ep) \cap L^2(0,T; W_\ep)\,,\\
  \label{mu_nloc}
  &\mu_\ep \in L^2(0,T; W)\,,\\
  \label{xi_nloc}
  &\xi_\ep \in L^2(0,T; H)\,, \qquad
  \xi_\ep\in\gamma(u_\ep) \quad\text{a.e.~in } Q\,,\\
  \label{eq1_nloc}
  &\partial_t u_\ep - \Delta \mu_\ep = 0
  \quad\text{a.e.~in } Q\,,\\
  \label{eq2_nloc}
  &\mu_\ep=\tau_\ep \partial_t u_\ep
  + B_\ep(u_\ep) +\xi_\ep +\Pi(u_\ep) -g_\ep
  \quad\text{a.e.~in } Q\,,\\
  \label{init_nloc}
  &u_\ep(0)=u_{0,\ep} \quad\text{a.e.~in } \Omega\,.
  \end{align}
  Furthermore, there exists a positive constant $M_\ep$
  such that, for every sets of data
  $(u_{0,\ep}^1, g_\ep^1)$ and $(u_{0,\ep}^2,g_\ep^2)$
  satisfying \eqref{ip1_nloc}--\eqref{ip2_nloc}, with 
  $(u_{0,\ep}^1)_\Omega=(u_{0,\ep}^2)_\Omega$,
  and for every respective solutions
  $(u_\ep^1, \mu_\ep^1, \xi_\ep^1)$
  and $(u_\ep^2, \mu_\ep^2, \xi_\ep^2)$
  satisfying \eqref{u_nloc}--\eqref{init_nloc}, it holds
  \begin{align*}
  &\|u_\ep^1-u_\ep^2\|_{C^0([0,T]; V^*)}^2+\tau_\ep\|u_\ep^1-u_\ep^2\|_{C^0([0,T]; H)}^2
  +\|E_\ep(u_\ep^1-u_\ep^2)\|_{L^1(0,T)}\\
  &\qquad\leq M_\ep
  \left(\|u_{0,\ep}^1-u_{0,\ep}^2\|_{V^*}^2+\tau_\ep \|u_{0,\ep}^1-u_{0,\ep}^2\|_{H}^2
  +\|g_\ep^1-g_\ep^2\|_{L^2(0,T; V^*)}^2\right)\,.
  \end{align*}
  In particular, the solution component
  $u_\ep$ is unique, and the solution 
  components $\mu_\ep$ and $\xi_\ep$
  are unique if $\gamma$ is single-valued.
\end{theorem}

Our second contribution concerns the 
nonlocal-to-local convergence. 
In particular, we show that, under suitable assumptions 
on the initial data $(u_{0,\ep})_{\ep}$ and on the
forcing terms $(g_\ep)_\ep$,
if the viscosities $(\tau_\ep)_\ep$ 
converge to a coefficient $\tau\geq0$, then the
solutions to the respective viscous nonlocal Cahn-Hilliard equations converge, in suitable topologies, 
to the solutions to the
limiting local Cahn-Hilliard equation
with viscosity parameter $\tau\geq0$.
Note that the viscosities $(\tau_\ep)_\ep$
are required to be strictly positive for all $\ep>0$,
whereas the limiting viscosity parameter $\tau$
is also allowed to vanish.
Hence, 
such result has a duplex formulation.
Indeed, if $\tau>0$ this shows the asymptotic convergence
of the nonlocal {\em viscous} equation to the corresponding
local {\em viscous} equation, while if $\tau=0$
this proves the approximability of solutions to 
the local {\em pure} equation by 
solutions to nonlocal equations with vanishing 
viscosities.

\begin{theorem}\label{th:conv}
Assume that
\[
\tau\geq0\,,\qquad 
(\tau_\ep)_{\ep>0}\subset(0,+\infty)\,,
\qquad
\lim_{\ep\searrow0}\tau_\ep=\tau\,.
\]
Let the data 
$(u_0,g)$ satisfy \eqref{ip1_loc}--\eqref{ip2_loc},
and let 
the family $(u_{0,\ep}, g_\ep)_{\ep>0}$
satisfy \eqref{ip1_nloc}--\eqref{ip2_nloc}
for all $\ep>0$.
Assume also that there exists
$\ep_0>0$ such that
\begin{align}
    \label{ip1_conv}
    &\sup_{\ep\in(0,\ep_0)}\left(
    \|u_{0,\ep}\|_{V_\ep}^2 + \|\hat\gamma(u_{0,\ep})\|_{L^1(\Omega)}
    \right) <+\infty\,,\\
    \label{ip1'_conv}
    &(g_\ep)_{\ep\in(0,\ep_0)}\subset
    H^1(0,T; H) \quad\text{and}\quad
    \sup_{\ep\in(0,\ep_0)}
    \|g_\ep\|_{H^1(0,T; H)}^2<+\infty \quad\text{if } \tau=0\,,\\
    \label{ip2_conv}
    &\exists\,[a_0,b_0]\subset
    \operatorname{Int}D(\gamma):\quad
    a_0\leq(u_{0,\ep})_\Omega\leq b_0
    \quad\forall\,\ep\in(0,\ep_0)\,,\\
    \label{ip3_conv}
    &u_{0,\ep}\rightharpoonup u_0 \quad\text{in } H
    \quad\text{as } \ep\searrow0\,,\qquad
    g_\ep\rightharpoonup g \quad\text{in } L^2(0,T; H)
    \quad\text{as } \ep\searrow0\,.
\end{align}
%Assume that $\tau_\ep\to \tau\in [0,+\infty)$, 
%\begin{equation}
%\label{eq:hp-gep}
%   \sup_{\ep>0} \frac{1}{\tau_\ep}\|g_\ep\|_{L^2(0,T;H)}<+\infty,
%\end{equation}
%and
%\begin{equation}
%    \label{eq:limit-g}
%    g_\ep\rightharpoonup g\quad\text{weakly in}\quad L^2(0,T;H).
%\end{equation}
%\end{comment}
Let 
$(u_\ep,\mu_\ep,\xi_\ep)_{\ep\in(0,\ep_0)}$
be a family of solutions to 
\eqref{u_nloc}--\eqref{init_nloc}
corresponding to 
the data $(u_{0,\ep}, g_\ep)$
and viscosity $\tau_\ep$,
where $u_\ep$ is uniquely 
determined.
Then, there exists a solution
$(u,\mu,\xi)$
to \eqref{u_loc}--\eqref{init_loc}
corresponding to the data $(u_0,g)$ and viscosity $\tau$,
where $u$ is uniquely determined,
such that, as $\ep\searrow0$,
\begin{align*}
    u_\ep \to u \qquad&\text{in } 
    C^0([0,T]; H)\,,\\ 
    \partial_t u_\ep \rightharpoonup \partial_t u \qquad&\text{in }
    L^2(0,T;V^*)\,,\\
    \partial_t u_\ep \rightharpoonup 
    \partial_t u \qquad&\text{in }
    L^2(0,T;H) \quad\text{if } \tau>0\,,\\
    \tau_\ep\partial_t u_\ep \to 0 \qquad&\text{in }
    L^2(0,T;H) \quad\text{if } \tau=0\,,\\
    \mu_\ep \rightharpoonup \mu \qquad&\text{in }
    L^2(0,T; V)\,,\\
    \mu_\ep \rightharpoonup \mu \qquad&\text{in }
    L^2(0,T; W) \quad\text{if } \tau>0\,,\\
    \xi_\ep \rightharpoonup \xi \qquad&\text{in }
    L^2(0,T; H)\,.
\end{align*}
\end{theorem}

%%%%%%%%%%%%%%%%%%%%%%%%%%%%%%%%%%%%%%%%%%%%%%%%%%%%
\section{Proof of Theorem~\ref{th:wp_nloc}}
\label{proof1}
%%%%%%%%%%%%%%%%%%%%%%%%%%%%%%%%%%%%%%%%%%%%%%%%%%%%

This section is devoted to the proof of well-posedness
of the nonlocal viscous Cahn-Hilliard equation.
Throughout the section, $\ep>0$ 
and $\tau_\ep>0$ are fixed.

\subsection{Approximation}
\label{subs:approx}
For every $\lambda>0$, 
let $\gamma_\lambda:\mathbb{R}\to
\mathbb{R}$ be the Yosida approximation of $\gamma$, having Lipschitz constant $1/\lambda$,
and set 
$\hat\gamma_\lambda(s):=\int_0^s\gamma_\lambda(r)\,\xd r$ for every $s\in \mathbb{R}$.
We consider the approximated problem
\begin{align}
    \label{1_app}
    \partial_t u_\ep^\lambda - \Delta \mu_\ep^\lambda
     = 0
    \qquad&\text{in } Q\,,\\
    \label{2_app}
    \mu_\ep^\lambda = \tau_\ep\partial_t u_\ep^\lambda-\lambda\Delta u_\ep^\lambda
    +B_\ep(u_\ep^\lambda)
    +\gamma_\lambda(u_\ep^\lambda)
    +\Pi(u_\ep^\lambda)-g_\ep
    \qquad&\text{in } Q\,,\\
    \label{4_app}
    \partial_{\bf n}u_\ep^\lambda=
    \partial_{\bf n}\mu_\ep^\lambda = 0
    \qquad&\text{in } \Sigma\,,\\
    \label{3_app}
    u_\ep^\lambda(0)=u_{0,\ep}^\lambda \qquad&\text{in } \Omega\,,
\end{align}
where the 
initial datum $u_{0,\ep}^\lambda$ satisfies
\begin{align}
    \label{ip_u0_ep_lam}
    &u_{0,\ep}^\lambda \in V\,, \qquad
    u_{0,\ep}^\lambda\to u_{0,\ep} \quad\text{in } H
    \quad\text{as } \ep\searrow0\,,\\
    \label{ip_u0_ep_lam2}
    &\sup_{\lambda\in(0,\lambda_0)}\left(
    \lambda\|u_{0,\ep}^\lambda\|_{V}^2+
    \|\hat\gamma(u_{0,\ep}^\lambda)\|_{L^1(\Omega)}
    %+E_\ep(u_{0,\ep}^\lambda)
    \right)<+\infty
\end{align}
for a certain $\lambda_0>0$
(possibly depending on $\ep$).
The existence of an approximating sequence $(u_{0,\ep}^\lambda)_\lambda$
satisfying \eqref{ip_u0_ep_lam}--\eqref{ip_u0_ep_lam2} is 
guaranteed by assumption \eqref{ip1_loc}:
for example,
one can check that the classical 
elliptic regularization given by the unique solution to 
the problem
\[
  \begin{cases}
  u_{0,\ep}^\lambda - \lambda\Delta u_{0,\ep}^\lambda = u_{0,\ep} \quad&\text{in } \Omega\,,\\
  \partial_{\bf n}u_{0,\ep}^\lambda=0
  \quad&\text{in } \partial\Omega\,,
  \end{cases}
\]
is a possible choice.
The existence of a unique approximated solution $(u_\ep^\lambda,\mu_\ep^\lambda)$ for every $\lambda>0$ relies on a fixed-point argument, as in
\cite[Section 3.1]{DRST}.
For every $v\in L^2(0,T; W)$, 
since $W\embed C^{0,\frac14}(\overline\Omega)$
by the Sobolev embeddings, thanks to 
the properties of $B_\ep$ proved in Lemma~\ref{lemma:properties} we have that 
$B_\ep(v)\in L^2(0,T; H)$.
Hence, by the classical literature on the local
viscous Cahn-Hilliard equation
(see again \cite{col-gil-spr}), 
the map
\[
  \Gamma_{\ep}^{\lambda}: C^0([0,T]; H)\cap L^2(0,T; W)
  \to H^1(0,T; H)\cap L^\infty(0,T; V)\cap L^2(0,T; W)\,,
  \quad \Gamma_\ep^\lambda:v\mapsto v_\ep^\lambda\,,
\]
is well-defined,
where $(v_\ep^\lambda, w_\ep^\lambda)$
is the unique 
solution
to the local viscous Cahn-Hilliard equation
\begin{align*}
    \partial_t v_\ep^\lambda - \Delta w_\ep^\lambda
     = 0
    \qquad&\text{in } Q\,,\\
    w_\ep^\lambda = \tau_\ep\partial_t v_\ep^\lambda-\lambda\Delta v_\ep^\lambda
    +\gamma_\lambda(v_\ep^\lambda)
    +\Pi(v_\ep^\lambda)-\left(g_\ep - B_\ep(v)\right)
    \qquad&\text{in } Q\,,\\
    \partial_{\bf n}u_\ep^\lambda=
    \partial_{\bf n}\mu_\ep^\lambda = 0
    \qquad&\text{in } \Sigma\,,\\
    v_\ep^\lambda(0)=u_{0,\ep}^\lambda \qquad&\text{in } \Omega\,.
\end{align*}
Now, arguing as in \cite[Section 3.1]{DRST},
exploiting the Lipschitz-continuity of $\gamma_\lambda$,
the Sobolev embeddings, and 
the properties of $B_\ep$ contained in 
Lemma~\ref{lemma:properties},
we deduce that 
there exist constants $L_\ep^\lambda>0$ 
and $\sigma>0$
such that,
for every $v_1,v_2\in C^0([0,T]; H)\cap
L^2(0,T; W)$, we have
\[
  \|\Gamma_\ep^\lambda(v_1)
  -\Gamma_\ep^\lambda(v_2)\|_{C^0([0,T]; H)
  \cap L^2(0,T; W)}
  \leq L_\ep^\lambda T^\sigma
  \|v_1-v_2\|_{L^2(0,T; W)}\,.
\]
It follows that one can choose $T_0\in(0,T]$
sufficiently small so that 
$\Gamma_\ep^\lambda$ is a contraction
on the respective functional spaces defined in $(0,T_0)$.
Performing then a classical patching argument
(we refer again to \cite[Section 3.1]{DRST} for details), 
we infer that $\Gamma_\ep^\lambda$ 
has a unique fixed point on the whole interval 
$[0,T]$.
This proves that the approximated system 
\eqref{1_app}--\eqref{3_app} has
a unique solution
\[
u_\varepsilon^\lambda\in H^1(0,T;H)
\cap L^\infty(0,T; V)\cap 
L^2(0,T; W)\,,\qquad
\mu_\varepsilon^\lambda \in L^2(0,T;W)\,.
\]

\subsection{Uniform estimates}
\label{s:unif_est}
We prove here some uniform estimates 
independently of $\lambda$ and $\ep$. In what follows we will always assume that $\lambda\in [0,1]$.
Moreover, $\ep>0$ and $\tau_\ep>0$
are still fixed.\\

 We start by fixing $t\in [0,T]$, testing \eqref{1_app} with  $\mu_\varepsilon^{\lambda}$, \eqref{2_app} with $\partial_t u_\varepsilon^{\lambda}$, taking the difference, and 
integrating the resulting equation on $(0,t)$.
We obtain
\begin{align*}
&\int_{Q_t}|\nabla\mu_\varepsilon^\lambda(s,x)|^2\,\xd x \,\xd s
+\tau_\ep \int_{Q_t}
|\partial_t u_\ep^\lambda(s,x)|^2\,\xd x \,\xd s\\
&\qquad+\frac\lambda2\int_\Omega|\nabla u_\ep^\lambda(t,x)|^2\,\xd x
+ E_{\ep}(u^\lambda_\varepsilon(t,\cdot))
+ \int_\Omega (\hat\gamma_\lambda+\hat\Pi)(u^\lambda_\varepsilon(t,x))\,\xd x \\
&\leq \, \frac\lambda2\int_\Omega
|\nabla u_{0,\ep}^\lambda(x)|^2\,\xd x+E_{\ep}(u_{0,\ep}^\lambda)+
\int_\Omega 
(\hat\gamma_\lambda+\hat\Pi)(u_{0,\ep}^\lambda(x))\,\xd x
+\int_{Q_t} |g_\ep(s,x)||\partial_t u(s,x)|\xd x\,\xd s.
\end{align*}
From the fact that
\[
\int_\Omega \hat{\gamma}_{\lambda}(u_{0,\ep}^\lambda(x))\,\xd x\leq \int_\Omega \hat{\gamma}(u_{0,\ep}^\lambda(x))\,\xd x\quad\text{for every}\,\lambda>0,
\]
using 
the uniform bound \eqref{ip_u0_ep_lam2} 
as well as the Young inequality, 
we get
\begin{align}
&\notag\int_{Q_t}
|\nabla\mu_\varepsilon^\lambda(s,x)|^2\,\xd x\,\xd s
+\frac{\tau_\ep}{2}\int_{Q_t}
|\partial_t u_\ep^\lambda(s,x)|^2\,\xd x\,\xd s+ E_{\ep}(u_\varepsilon^\lambda(t,\cdot))
 +\frac\lambda2\int_\Omega|\nabla u_\ep^\lambda(t,x)|^2\,\xd x\\
&\label{eq:est1-high}\quad\leq  C_\ep
+\frac{\tau_\ep}4 \int_{Q_t} 
|\partial_t u^\lambda_\varepsilon(t,x)|^2\, \xd x\,\xd t 
+  \frac{1}{\tau_\ep}\int_0^T \int_{\Omega}|g_\ep(t,x)|^2\,\xd x\, \xd t 
\end{align}
for every $t\in [0,T]$, where $C_\ep>0$ is a constant
independent of $\lambda$ and depending only on the initial 
datum $u_{0,\ep}$.

From the arbitrariness of $t\in [0,T]$ we deduce that,
for every $\lambda\in(0,1)$,
\begin{align}
&\|\nabla\mu_\varepsilon^\lambda \|_{L^2(0,T;H)}
\leq C_\ep\,, \label{grad_mu} \\
&\|u_\varepsilon^\lambda\|_{L^\infty(0,T;V_\ep)\,}+
\|u_\varepsilon^\lambda\|_{H^1(0,T;H)}+\lambda^{1/2}\|\nabla u_\ep^\lambda\|_{L^\infty(0,T;H)}
\leq C_\ep\,,
\label{u-linfty} 
\end{align}
hence also, by comparison in \eqref{1_app},
\begin{equation}
  \label{mu_delta}
    \|\Delta\mu_\ep^\lambda\|_{L^2(0,T; H)}\leq C_\ep\,.
\end{equation}

Furthermore, noting that $(u_\ep^\lambda)_\Omega=
(u_{0,\ep}^\lambda)_\Omega=(u_{0,\ep})_\Omega$, 
we test \eqref{1_app} by $\mathcal N(u_\ep^\lambda - 
(u_{0,\ep})_\Omega)$, \eqref{2_app} by 
$u_\ep^\lambda - (u_{0,\ep})_\Omega$, and sum:
we obtain, for almost every $t\in(0,T)$,
\begin{align*}
    &\langle\partial_t u_\ep^\lambda(t),
    \mathcal N(u_\ep^\lambda(t) - (u_{0,\ep})_\Omega)\rangle_{V}
    +\tau_\ep\langle
    \partial_t u_\ep^\lambda(t), 
    u_\ep^\lambda(t) - (u_{0,\ep})_\Omega\rangle_V
    +\lambda\int_\Omega|\nabla u_\ep^\lambda(t,x)|^2\xd x\\
    &\qquad+\int_\Omega B_\ep(u_\ep^\lambda)(t,x)u_\ep^\lambda(t,x)\,\xd x
    +\int_\Omega\gamma_\lambda(u_\lambda(t,x))
    (u_\ep^\lambda(t,x) - (u_{0,\ep})_\Omega)\,\xd x\\
    &=\int_\Omega\left(g^\ep(t,x) - 
    \Pi(u_\ep^\lambda)(t,x)\right)
    (u_\ep^\lambda(t,x) - (u_{0,\ep})_\Omega)\,\xd x\,,
\end{align*}
where we have used that 
$\int_\Omega B_\ep(u_\ep^\lambda(t,x))\,\xd x
=0$ by the symmetry of the kernel $K_\ep$.
A classical argument shows that since
$(u_{0,\ep})_\Omega \in \operatorname{Int}D(\gamma)$,
then there are two constants $c_\ep, c_\ep'$,
only depending on the position of $(u_{0,\ep})_\Omega$,
such that
\[
  \|\gamma_\lambda(u_\ep^\lambda(t,\cdot))
  \|_{L^1(\Omega)}\leq
  c_\ep\int_\Omega\gamma_\lambda
  (u_\ep^\lambda(t,x))
  (u_\ep^\lambda(t,x)-(u_{0,\ep})_\Omega)
  \,\xd x +c_\ep'\,, \quad\text{for
  a.e.~}t\in(0,T)\,.
\]
Arguing as in \cite[Subsection 3.2]{DRST}, 
the estimates above and \eqref{grad_mu}--\eqref{u-linfty}
yield then
a control on 
$\|\gamma_\lambda(u_\lambda)\|_{L^2(0,T; L^1(\Omega))}$.
In particular, 
by comparison in \eqref{2_app} 
we get an 
estimate on $(\mu_\ep^\lambda)_\Omega$ in $L^2(0,T)$.
Taking \eqref{grad_mu} and \eqref{mu_delta}
into account, we deduce then that
\begin{align}
\Vert\mu_\varepsilon^\lambda\Vert_{L^2(0,T;W)} & \leq C_\ep\,. \label{est_mu}
\end{align}
By comparison in \eqref{2_app}
we infer that 
\[
  \|-\lambda\Delta u_\ep^\lambda
  +B_\ep(u_\ep^\lambda) + 
  \gamma_\lambda(u_\ep^\lambda)\|_{L^2(0,T; H)}\leq C_\ep\,.
\]
Testing $-\lambda\Delta u_\ep^\lambda
+B_\ep(u_\ep^\lambda) + 
\gamma_\lambda(u_\ep^\lambda)$
by $\gamma_\lambda(u_\ep^\lambda)$
and noting that, by monotonicity of $\gamma_\lambda$,
\begin{align*}
  &\int_\Omega (-\lambda\Delta u_\ep^\lambda(t,x)
  +B_\ep(u_\ep^\lambda)(t,x))
  \gamma_\lambda(u_\ep^\lambda)(t,x)\,\xd x\\
  &=\lambda\int_\Omega\gamma_\lambda'(u_\ep^\lambda)
  |\nabla u_\ep^\lambda(t,x)|^2\,\xd x\\
  &\qquad+\frac12\int_\Omega\int_\Omega
  K_\ep(x,y)\left(\gamma_\lambda(u_\ep^\lambda(t,x))
  -\gamma_\lambda(u_\ep^\lambda(t,y))\right)
  \left(u_\ep^\lambda(t,x)-u_\ep^\lambda(t,y)\right)
  \,\xd x\,\xd y
  \geq0\,,
\end{align*}
by the estimate above and the Young inequality we 
also deduce that
\begin{equation}
\label{est_F'}
\|-\lambda\Delta u_\ep^\lambda
  +B_\ep(u_\ep^\lambda)\|_{L^2(0,T; H)} + 
  \|\gamma_\lambda(u_\ep^\lambda)\|_{L^2(0,T; H)}\leq C_\ep\,.
\end{equation}

%%%%%%%%%%%%%%%%%%%%%%%%%%%%%%%%%%%%%%%%%%%%%%%%%%%%%%%%%%%%%%%%%%
\subsection{Passage to the limit as $\lambda\searrow0$}
\label{subs:conv}
%%%%%%%%%%%%%%%%%%%%%%%%%%%%%%%%%%%%%%%%%%%%%%%%%%%%%%%%%%
In this section we analyze the passage 
to the limit as $\lambda\searrow0$, 
with $\ep>0$  and $\tau_\ep>0$ still fixed.
In view of the uniform bounds 
\eqref{grad_mu}--\eqref{est_F'}
and the Aubin-Lions lemma,
up to the extraction of (not relabeled) subsequences
we have the following convergences:
\begin{align}
 \label{conv_strong_lam}
 u_{\ep}^\lambda&\to u_\ep
 \quad&\text{in } C^0([0,T];V^*)\,,\\
 u_\varepsilon^\lambda
 &\wstarto u_\ep 
 &\text{in } L^\infty(0,T;V_\ep)\cap H^1(0,T;H)\,, \label{conv_u_lam}\\
 \lambda u_\ep^\lambda &\to 0
 &\text{in } L^\infty(0,T; V)\,,
 \label{conv_u_lam_V}\\
 \mu_\varepsilon^\lambda &\rightharpoonup \mu_\ep 
 & \text{in } L^2(0,T;W)\,,\label{conv_mu_lam}\\
 \gamma_\lambda(u^\lambda_\varepsilon) 
 &\rightharpoonup \xi_{\ep} & \text{in } L^2(0,T;H)\,, \label{conv_gamma_lam}\\
 \Pi(u^\lambda_\varepsilon) 
 &\rightharpoonup \Xi_{\ep} & \text{in } L^2(0,T;H)\,, \label{conv_pi_lam}\\
 -\lambda\Delta u_\ep^\lambda + 
 B_\ep(u_\ep^\lambda) &\rightharpoonup \eta_\ep
 &\text{in } L^2(0,T; H)\,,
 \label{conv_last_lam}
\end{align}
for some
\begin{gather*}
u_\ep\in H^1(0,T;H)\cap L^\infty(0,T;V_\ep)\,,\qquad
\mu_\ep \in L^2(0,T;W)\,, \\
\xi_{\ep} \in L^2(0,T;H)\,, \qquad
\Xi_\ep\in L^2(0,T; H)\,, \qquad \eta_\ep\in L^2(0,T; H)\,.
\end{gather*}
From \eqref{conv_u_lam} and the fact that $B_\ep\in\mathscr{L}
(V_\ep, V_\ep^*)$, it is readily seen that 
\[
  B_\ep(u_\ep^\lambda)\wstarto B_\ep(u_\ep)
  \qquad\text{in } L^\infty(0,T; V_\ep^*)\,.
\]
Moreover, from \eqref{conv_u_lam_V} and 
\eqref{conv_last_lam}, it follows by comparison
that 
\[
  B_\ep(u_\ep^\lambda)\rightharpoonup \eta_\ep
  \qquad\text{in } L^2(0,T; V^*)\,.
\]
We deduce in particular that 
$B_\ep(u_\ep)=\eta_\ep\in L^2(0,T; H)$, so that also
$u_\ep\in L^2(0,T; W_\ep)$.
The strong convergence \eqref{conv_strong_lam}
implies also that $u_\ep(0)=u_{0,\ep}$. 

Passing to the limit in \eqref{1_app}-\eqref{3_app} in the weak topology of $L^2(0,T;H)$, we obtain
\begin{align}
    \label{1_app-ep}
    \partial_t u_\ep - \Delta \mu_\ep
     = 0 \quad&\text{in } L^2(0,T; H)\,,\\
    \label{2_app-ep}
    \mu_\ep = \tau_\ep\partial_t u_\ep
    +{B}_\ep(u_\ep)
    +\xi_{\ep}
    +\Xi_{\ep}-g_\ep
    \quad&\text{in } L^2(0,T; H)\,,\\
    \label{4_app-ep}
    \partial_{\bf n}\mu_\ep = 0 \quad&\text{in } L^2(\Sigma)\,,\\
    \label{3_app-ep}
    u_\ep(0)=u_{0,\ep} \quad&\text{in } H\,.
\end{align}
We proceed now providing an identification of
the nonlinear terms $\xi_\ep$ and $\Xi_\ep$:
we adapt an argument performed in
\cite[Subsection~3.6]{col-scar}.
To this end, since $\Pi$ is Lipschitz-continuous,
there exists $\alpha>0$ such that the operator
\[
\gamma+\Pi+\alpha
\tau_\ep\,\mathrm{Id}:\mathbb{R}\to 2^{\mathbb{R}}
\]
is maximal monotone. For example, one can choose
$\alpha:=\frac{2}{\tau_\ep}\|\Pi'\|_{L^\infty(\mathbb R)}$
(recall that $\tau_\ep>0$ is fixed).
    Multiplying \eqref{2_app} by $e^{-\alpha t}$, we obtain
    $$e^{-\alpha t}\mu_{\ep}^\lambda=\tau_\ep \partial_t (e^{-\alpha t}u_\ep^\lambda)-\lambda \Delta (e^{-\alpha t}u_\ep^\lambda)+B_\ep(e^{-\alpha t}u_\ep^\lambda)
    +e^{-\alpha t}
    (\gamma_\lambda(u_\ep^\lambda)
    +\Pi(u_\ep^\lambda)+\alpha\tau_\ep u_\ep^\lambda - g_\ep).$$
Thus, testing the previous 
equation by $e^{-\alpha t}u_\ep^\lambda$
and integrating in time yields
\begin{align*}
&\limsup_{\lambda\to 0}\int_Q 
e^{-2\alpha s}(\gamma_\lambda(u_\ep^\lambda(s,x))
+\Pi(u_\ep^\lambda(s,x))
+\alpha\tau_\ep u_\ep^\lambda(s,x))u_\ep^\lambda(s,x)
\, \xd x\,\xd s\\
&\leq \limsup_{\lambda\to 0}\left[\int_Q 
e^{-2\alpha s}\mu_\ep^\lambda(s,x) 
u_\ep^\lambda(s,x)\,\xd x\,\xd s
-\lambda\int_Q e^{-2\alpha s}
|\nabla u_\ep^\lambda(s,x)|^2\,\xd x\,\xd s\right.\\
&\qquad\qquad\left.-\frac{\tau_\ep}{2}\int_\Omega 
e^{-2\alpha T}|u_\ep^\lambda(T,x)|^2\,\xd x
+\frac{\tau_\ep}{2}\int_\Omega|u_{0,\ep}^\lambda(x)|^2\,\xd x
-2\int_0^T e^{-2\alpha s}E_\ep(u_\ep^\lambda(s,\cdot))\,\xd s\right.\\
&\qquad\qquad\left.+\int_Q e^{-2\alpha s} g_\ep(s,x)u_\ep^\lambda(s,x)\,\xd x\,\xd s\right].
\end{align*}

On the one hand, owing to \eqref{conv_strong_lam} and \eqref{conv_mu_lam},
$$
\lim_{\lambda\to 0}\int_Q e^{-2\alpha s}
(\mu_\ep^\lambda(s,x)+g_\ep(s,x))
u_\ep^\lambda(s,x)\,\xd x\,\xd s
=\int_Q 
e^{-2\alpha s}(\mu_\ep(s,x)+g_\ep(s,x))
u_\ep(s,x)\,\xd x\,\xd s\,.
$$
On the other hand, by the weak
lower semicontinuity of the norms,
the convergence \eqref{conv_u_lam},
and the assumption \eqref{ip_u0_ep_lam}, we have
\begin{align*}
&\limsup_{\lambda\to 0}
\left[-\lambda\int_Q 
e^{-2\alpha s}|\nabla u_\ep^\lambda(s,x)|^2\,\xd x\,\xd s\right.\\
&\qquad\qquad\left.-\frac{\tau_\ep}{2}\int_\Omega
e^{-2\alpha T}|u_\ep^\lambda(T,x)|^2\,\xd x
+\frac{\tau_\ep}{2}\int_\Omega|u_{0,\ep}^\lambda(x)|^2\,\xd x
-2\int_0^T e^{-2\alpha s}
E_\ep(u_\ep^\lambda(s,\cdot))\,\xd s\right]\\
&\leq-\frac{\tau_\ep}{2}\liminf_{\lambda\to 0}
\int_\Omega e^{-2\alpha T}|u_\ep^\lambda(t,x)|^2\,\xd x
+\frac{\tau_\ep}{2}\limsup_{\lambda\to0}
\int_\Omega|u_{0,\ep}^\lambda(x)|^2\,\xd x\\
&\qquad\qquad-2\liminf_{\lambda\to0}
\int_0^T e^{-2\alpha s}
E_\ep(u_\ep^\lambda(s,\cdot))\,\xd s\\
&\leq 
-\frac{\tau_\ep}{2}\int_\Omega 
e^{-2\alpha T}|u_\ep(T,x)|^2\,\xd x
+\frac{\tau_\ep}{2}\int_\Omega|u_{0,\ep}(x)|^2\,\xd x
-2\int_0^T e^{-2\alpha s}E_\ep(u_\ep(s,\cdot))\,\xd s\,.
\end{align*} 

Hence, we deduce that
\begin{align}
&\nonumber
\limsup_{\lambda\to 0}\int_Q 
e^{-2\alpha s}
(\gamma_\lambda(u_\ep^\lambda(s,x))+
\Pi(u_\ep^\lambda(s,x))+
\alpha\tau_\ep u_\ep^\lambda(s,x))u_\ep^\lambda(s,x)
\, \xd x\,\xd s\\
&\nonumber\leq \int_Q 
e^{-2\alpha s}(\mu_\ep(s,x)+g_\ep(s,x))u_\ep(s,x)\,\xd x\,\xd s\\
&\label{eq:limsup-arg}\qquad
-\frac{\tau_\ep}{2}\int_\Omega 
e^{-2\alpha T}|u_\ep(T,x)|^2\,\xd x
+\frac{\tau_\ep}{2}\int_\Omega|u_{0,\ep}(x)|^2\,\xd x
-2\int_0^T e^{-2\alpha s}E_\ep(u_\ep(s,\cdot))\,\xd s\,.
\end{align}
Testing \eqref{2_app-ep} by $e^{-2\alpha t}u_\ep$ 
and integrating in time, the right-hand side of \eqref{eq:limsup-arg} rewrites as
\begin{align*}
&\nonumber\limsup_{\lambda\to 0}\int_Q 
e^{-2\alpha s}
(\gamma_\lambda(u_\ep^\lambda(s,x))
+\Pi(u_\ep^\lambda(s,x))+\alpha\tau_\ep u_\ep^\lambda(s,x))u_\ep^\lambda(s,x) \xd x\,\xd s\\
&\nonumber\leq \int_0^t\int_\Omega 
e^{-2\alpha s}(\xi_{\ep}(s,x)+\Xi_{\ep}(s,x)+
\alpha\tau_\ep u_{\ep}(s,x))u_{\ep}(s,x)\,\xd x\,\xd s\,.
\end{align*}
Since the bilinear form
\[
  (v_1,v_2)\mapsto \int_Qe^{-2\alpha x}
  v_1(s,x)v_2(s,x)\,\xd x\,\xd s\,,
  \quad v_1,v_2\in L^2(Q)\,,
\]
is an equivalent scalar product on $L^2(Q)$,
by the maximal monotonicity of 
$\gamma+\Pi+\alpha\tau_\ep\,\mathrm{Id}$ we conclude that
\begin{equation}\label{incl_alpha}
\xi_{\ep}+\Xi_{\ep}+\alpha\tau_\ep u_{\ep}\in (\gamma+\Pi+\alpha\tau_\ep\,\mathrm{Id})(u_\ep)
\quad\text{a.e.~in } Q\,.
\end{equation}
This allows us to show the further strong convergences:
\begin{equation}
    \label{strong_H}
    u_\ep^\lambda(t) \to u_\ep (t)
    \quad\text{in } H \quad\forall\,t\in[0,T]\,,
    \qquad
    u_\ep^\lambda \to u_\ep
    \quad\text{in } L^2(0,T; V_\ep)\,.
\end{equation}
Indeed, taking the difference between
\eqref{2_app} and \eqref{2_app-ep}, multiplying 
again by $e^{-\alpha t}$, and testing by 
$e^{-\alpha t}(u_\ep^\lambda-u_\ep)$, we get
\begin{align*}
    &\frac{\tau_\ep}{2}\int_\Omega e^{-2\alpha t}
    |(u_\ep^\lambda-u_\ep)(t,x)|^2\,\xd x
    +\lambda\int_{Q_t}e^{-2\alpha s}|\nabla
    u_\ep^\lambda(s,x)|^2\,\xd x\,\xd s\\
    &\qquad+2\int_0^Te^{-2\alpha s}
    E_\ep((u_\ep^\lambda-u_\ep)(s,x))\,\xd s\\
    &\qquad+\int_{Q_t}e^{-2\alpha s}
    \left(\gamma_\lambda(u_\ep^\lambda)
    +\Pi(u_\ep^\lambda)
    +\alpha\tau_\ep u_{\ep}^\lambda
    -(\xi_\ep + \Xi_\ep + \alpha\tau_\ep u_\ep)
    \right)(s,x)
    (u_\ep^\lambda-u_\ep)(s,x)\,\xd x\,\xd s\\
    &=\frac{\tau_\ep}2\int_\Omega
    |u_{0,\ep}^\lambda(x)- u_{0,\ep}(x)|^2\,\xd x
    +\int_{Q_t}e^{-2\alpha s}
    (\mu_\lambda-\mu)(s,x)(u_\ep^\lambda-u_\ep)(s,x)
    \,\xd x\,\xd s\\
    &\qquad-\lambda\int_{Q_t}e^{-2\alpha s}
    \Delta u_\ep^\lambda(s,x)u_\ep(s,x)
    \,\xd x\,\xd s\,.
\end{align*}
We use now the notation $J^\gamma_\lambda:=
(\mathrm{Id}+\lambda\gamma)^{-1}:\mathbb{R}\to
\mathbb{R}$ for the resolvent of $\gamma$.
Summing and subtracting $J_\lambda^\gamma(u_\ep^\lambda)$
in the last term on the left-hand side, 
rearranging the terms, and recalling that 
$u_\ep^\lambda-J_\lambda^\gamma(u_\ep^\lambda)=
\lambda\gamma_\lambda(u_\ep^\lambda)$, 
we infer that, for every $t\in[0,T]$,
\begin{align*}
    &\frac{\tau_\ep}{2}\int_\Omega e^{-2\alpha t}
    |(u_\ep^\lambda-u_\ep)(t,x)|^2\,\xd x
    +2\int_0^Te^{-2\alpha s}
    E_\ep((u_\ep^\lambda-u_\ep)(s,x))\,\xd s\\
    &+\int_{Q_t}e^{-2\alpha s}
    \left(\gamma_\lambda(u_\ep^\lambda)
    {+}\Pi(J^\gamma_\lambda(u_\ep^\lambda))
    {+}\alpha\tau_\ep J^\gamma_\lambda(u_{\ep}^\lambda)
    {-}(\xi_\ep {+} \Xi_\ep {+} \alpha\tau_\ep u_\ep)
    \right)(s,x)
    (J_\lambda(u_\ep^\lambda){-}u_\ep)(s,x)\,\xd x\,\xd s\\
    &\leq \frac{\tau_\ep}2\int_\Omega
    |u_{0,\ep}^\lambda(x)- u_{0,\ep}(x)|^2\,\xd x +
    \int_{Q_t}e^{-2\alpha s}
    (\mu_\lambda-\mu)(s,x)(u_\ep^\lambda-u_\ep)(s,x)
    \,\xd x\,\xd s\\
    &-\int_{Q_t}e^{-2\alpha s}
    B_\ep(u_\ep^\lambda(s,x)))u_\ep(s,x)
    \,\xd x\,\xd s\\
    &+\int_{Q_t}e^{-2\alpha s}
    (-\lambda\Delta u_\ep^\lambda(s,x)
    +B_\ep(u_\ep^\lambda(s,x)))u_\ep(s,x)
    \,\xd x\,\xd s\\
    &+\int_{Q_t}e^{-2\alpha s}
    \left(\Pi(J^\gamma_\lambda(u_\ep^\lambda
    (s,x))) {-} \Pi(u_\ep^\lambda(s,x))
    {+}\alpha\tau_\ep(J^\gamma_\lambda(u_\ep^\lambda){-}u_\ep^\lambda)(s,x)\right)
    (J_\lambda(u_\ep^\lambda){-}u_\ep)(s,x)\,\xd x\,\xd s\\
    &-\lambda\int_{Q_t}e^{-2\alpha s}
    \left(\gamma_\lambda(u_\ep^\lambda)
    +\Pi(u_\ep^\lambda)
    +\alpha\tau_\ep u_{\ep}^\lambda
    -(\xi_\ep + \Xi_\ep + \alpha\tau_\ep u_\ep)
    \right)(s,x)
    \gamma_\lambda(u_\ep^\lambda(s,x))\,\xd x\,\xd s
    \,.
\end{align*}
Recalling that 
$\gamma_\lambda(r)\in\gamma(J_\lambda^\gamma(r))$
for every $r\in \mathbb{R}$,
by \eqref{incl_alpha} and the monotonicity 
of the operator 
$\gamma + \Pi + \alpha\tau_\ep\mathrm{Id}$,
the third term on the left-hand side is 
nonnegative. Let us show that the right-hand side
converges to $0$, analyzing each term
separately.
The first two terms on the right-hand side
converge to $0$ thanks to \eqref{ip_u0_ep_lam},
\eqref{conv_strong_lam} and \eqref{conv_mu_lam}.
Moreover, thanks to \eqref{conv_u_lam},
\eqref{conv_last_lam},
and the fact that $u_\ep\in L^2(0,T; W_\ep)$, we have
\[
    -\int_{Q_t}e^{-2\alpha s}
    B_\ep(u_\ep^\lambda(s,x))u_\ep(s,x)
    \,\xd x\,\xd s \to 
    -\int_{Q_t}e^{-2\alpha s}
    B_\ep(u_\ep(s,x))u_\ep(s,x)
    \,\xd x\,\xd s
\]
and
\[
    \int_{Q_t}e^{-2\alpha s}
    (-\lambda\Delta u_\ep^\lambda(s,x)
    +B_\ep(u_\ep^\lambda(s,x)))u_\ep(s,x)
    \,\xd x\,\xd s \to 
    \int_{Q_t}e^{-2\alpha s}
    B_\ep(u_\ep(s,x))u_\ep(s,x)
    \,\xd x\,\xd s\,.
\]
Finally, since $(\gamma_\lambda(u_\ep^\lambda))_\lambda$
is bounded in $L^2(0,T; H)$ by \eqref{est_F'},
using the Lipschitz-continuity of $Pi$,
the last two terms on the right-hand side can be 
handled by
\begin{align*}
    &\lambda
    \|\gamma_\lambda(u_\ep^\lambda)\|_{L^2(0,T; H)}
    \left(\|J_\lambda^\gamma
    (u_\ep^\lambda)\|_{L^2(0,T; H)}
    +\|u_\ep\|_{L^2(0,T; H)}\|\right.\\
    &\qquad\left.+\|\gamma_\lambda(u_\ep^\lambda)
    +\Pi(u_\ep^\lambda)
    +\alpha\tau_\ep u_{\ep}^\lambda
    -(\xi_\ep + \Xi_\ep 
    + \alpha\tau_\ep u_\ep)\|_{L^2(0,T; H)}
    \right)
    \leq C_\ep \lambda \to 0\,.
\end{align*}
Since $t\in[0,T]$ is arbitrary, the strong convergences
\eqref{strong_H} follows.
In particular, this readily implies that 
$\Xi_\ep=\Pi(u_\ep)$
and $\xi_\ep \in \gamma (u_\ep)$
almost everywhere in $Q$
by the Lipschitz-continuity 
of $\Pi$ and
by the maximal monotonicity of $\gamma$, respectively.

It is then clear that 
$(u_\ep,\mu_\ep,\xi_\ep)$ is a solution to the nonlocal viscous Cahn-Hilliard equation in the sense of
\eqref{u_nloc}--\eqref{init_nloc}.
This completes the proof of 
the first assertion of Theorem \ref{th:wp_nloc}.

\subsection{Continuous dependence}
Let $(u_{0,\ep}^1, g_\ep^1)$ and 
$(u_{0,\ep}^2, g_\ep^2)$ satisfy the assumptions
\eqref{ip1_nloc}--\eqref{ip2_nloc}
with $(u_{0,\ep}^1)_\Omega=
(u_{0,\ep}^2)_\Omega$, and let 
$(u_\ep^1, \mu_\ep^1,\xi_\ep^1)$ and $(u_\ep^2,
\mu_\ep^2, \xi_\ep^2)$ be any corresponding 
solutions to \eqref{u_nloc}--\eqref{init_nloc}.

We observe that their difference solves
\begin{align*}
    \partial_t(u_\ep^1-u_\ep^2)-\Delta(\mu_\ep^1-\mu_\ep^2)=0
    \quad&\text{in } Q\,,\\
    \mu_\ep^1-\mu_\ep^2=\tau_\ep\partial_t (u_\ep^1-u_\ep^2)+B_\ep(u_\ep^1-u_\ep^2)
    +\xi_\ep^1-\xi_\ep^2 + \Pi(u_\ep^1)-\Pi(u_\ep^2)
    -(g_\ep^1-g_\ep^2)
    \quad&\text{in Q}\,,\\
    \partial_{\bf n}(\mu_\ep^1-\mu_\ep^2)=0
    \quad&\text{in }\Sigma\,,\\
    (u_\ep^1-u_\ep^2)(0)=0 \quad&\text{in } \Omega\,.
\end{align*}
By the assumption on the initial data, we have that $(u_\ep^1-u_\ep^2)_\Omega=0$. Therefore, we can
test the first equation by $\mathcal N(u_\ep^1-u_\ep^2)$,
the second by $u_\ep^1-u_\ep^2$, and take the difference:
by performing classical computations we get
\begin{align*}
    &\frac12\|(u_\ep^1-u_\ep^2)(t)\|_{V^*}^2+\frac{\tau_\ep}{2}\|(u_\ep^1-u_\ep^2)(t)\|_{H}^2
    +2\int_0^t E_\ep(u_\ep^1-u_\ep^2)(s)\,\xd s\\
    &\qquad +\int_{Q_t}
    (\xi_\ep^1-\xi_\ep^2)(s,x)(u_\ep^1-u_\ep^2)(s,x)
    \,\xd x\,\xd s\\
    &=\frac12\|(u_{0,\ep}^1
    -u_{0,\ep}^2)\|_{V^*}^2+\frac{\tau_\ep}{2}\|(u_{0,\ep}^1
    -u_{0,\ep}^2)\|_{H}^2\\
    &\qquad+\int_{Q_t}\left(g_\ep^1-g_\ep^2
    -\Pi(u_\ep^1)+\Pi(u_\ep^2)\right)(s,x)
    (u_\ep^1-u_\ep^2)(s,x)\,.
\end{align*}
The last term on the left-hand side is nonnegative
by the monotonicity of $\gamma$. Hence,
the continuous-dependence
property stated in Theorem~\ref{th:wp_nloc} follows from
the Lipschitz-continuity of $\Pi$ and the Gronwall lemma.

%%%%%%%%%%%%%%%%%%%%%%%%%%%%%%%%%%%%%%%%%%%%%%%%%%%%%%%
\section{Proof of Theorem~\ref{th:conv}}
\label{proof2}
%%%%%%%%%%%%%%%%%%%%%%%%%%%%%%%%%%%%%%%%%%%%%%%%%%%%%%%
This section is devoted to study the asymptotic behavior of solutions to the nonlocal viscous Cahn-Hilliard equation
as $\ep\searrow0$.

Let us recall that the family of data $(u_{0,\ep},g_\ep)_{\ep>0}$
are assumed to satisfy \eqref{ip1_conv}--\eqref{ip3_conv},
while $(u_\ep, \mu_\ep, \xi_\ep)$ is a corresponding 
solution to 
\eqref{u_nloc}--\eqref{init_nloc}.

\subsection{The case $\tau>0$}
\label{ss:tau_pos}
We consider here the case $\tau>0$, so that $\tau_\ep\to \tau>0$.
As a major consequence, this implies that 
it is not restrictive to assume that 
\begin{equation}\label{tau_bound}
    \exists\,\tau_*>0: \quad\tau_\ep\geq\tau_*  
    \quad\forall\,\ep\in(0,\ep_0)\,.
\end{equation}

We test \eqref{eq1_nloc} by $\mu_\ep$, 
\eqref{eq2_nloc} by $\partial_t u_\ep$, take the difference,
and integrate on $Q_t$: 
recalling \eqref{ip1'_conv} and using the Young inequality,
we deduce that 
\begin{align*}
&\int_{Q_t}|\nabla\mu_\varepsilon(s,x)|^2\,\xd x \,\xd s
+\tau_\ep \int_{Q_t}
|\partial_t u_\ep(s,x)|^2\,\xd x \,\xd s
+ E_{\ep}(u_\varepsilon(t,\cdot)) + 
\int_\Omega (\hat\gamma+\hat\Pi)
(u_\varepsilon(t,x))\,\xd x \\
&\leq E_{\ep}(u_{0,\ep})+
\int_\Omega 
(\hat\gamma+\hat\Pi)(u_{0,\ep}(x))\,\xd x
+\frac{\tau_\ep}{2}\int_{Q_t}|\partial_t u_\ep(s,x)|^2\,\xd x\,\xd s
+\frac1{\tau_\ep}\int_{Q_t}|g_\ep(s,x)|^2\,\xd x\,\xd s\,.
\end{align*}
Note that $\frac1{\tau_\ep}\leq\frac1{\tau_*}$
by \eqref{tau_bound}. Hence, 
rearranging the terms and 
using \eqref{ip1_conv} we infer that 
there exists a constant $C>0$, independent of $\ep$,
such that 
\[
  \|\nabla \mu_\ep\|_{L^2(0,T; H)} + 
  \|u_\ep\|_{H^1(0,T; H)\cap L^\infty(0,T; V_\ep)}\leq C\,
\]
hence also, by comparison in \eqref{eq1_nloc},
\[
    \|\Delta\mu_\ep\|_{L^2(0,T; H)}\leq C\,.
\]
Now, we can proceed as in the previous Section~\ref{s:unif_est}.
Since $(u_\ep)_\Omega=(u_{0,\ep})_\Omega$, 
we can test
\eqref{eq1_nloc} by $\mathcal N(u_\ep - (u_{0,\ep})_\Omega)$, \eqref{eq2_nloc} by $u_\ep - (u_{0,\ep})_\Omega$, and sum:
we obtain, for almost every $t\in(0,T)$,
\begin{align*}
    &\langle\partial_t u_\ep(t),
    \mathcal N(u_\ep(t) - (u_{0,\ep})_\Omega)\rangle_{V}
    +\tau_\ep\langle
    \partial_t u_\ep(t), 
    u_\ep(t) - (u_{0,\ep})_\Omega\rangle_V
    +2E_\ep(u_\ep(t,x))\\
    &\qquad+\int_\Omega\xi_\ep(t,x)
    (u_\ep(t,x) - (u_{0,\ep})_\Omega)\,\xd x\\
    &=\int_\Omega\left(g^\ep(t,x) - 
    \Pi(u_\ep)(t,x)\right)
    (u_\ep(t,x) - (u_{0,\ep})_\Omega)\,\xd x.
\end{align*}
Again, by the estimates already performed, all the terms
are bounded in $L^2(0,T)$ except
\[
  \int_\Omega\xi_\ep(t,x)
    (u_\ep(t,x) - (u_{0,\ep})_\Omega)\,\xd x\,.
\]
Thanks to assumption \eqref{ip2_conv},
there are two constants $c, c'>0$,
independent of $\ep$,
such that
\[
  \|\xi_\ep(t,\cdot)\|_{L^1(\Omega)}\leq
  c\int_\Omega\xi_\ep(t,\cdot)
  (u_\ep(t,x)-(u_{0,\ep})_\Omega)\,\xd x +c'\,.
\]
Hence, we deduce that 
\[
\|\xi_\ep\|_{L^2(0,T; L^1(\Omega))}\leq C\,,
\]
which implies, by comparison in \eqref{eq2_nloc},
that 
\[
\|(\mu_\ep)_\Omega\|_{L^2(0,T)}\leq C\,.
\]
We deduce that 
\[
  \|\mu_\ep\|_{L^2(0,T; W)}\leq C\,.
\]
Thus, by comparison in \eqref{eq2_nloc}
and by monotonicity of $\gamma$, we obtain that
\[
  \|B_\ep(u_\ep)\|_{L^2(0,T; H)} 
  + \|\xi_\ep\|_{L^2(0,T; H)}\leq C\,.
\]

By the Aubin-Lions compactness theorem we infer that, up to the extraction of (not relabeled) subsequences, 
as $\ep\searrow0$,
\begin{align}
 \label{strong_star}
 u_{\ep}\to u \qquad&\text{in }  C^0([0,T]; V^*)\,,\\
 u_\varepsilon
 \rightharpoonup u \qquad&\text{in } H^1(0,T;H)\,, \\
 B_\varepsilon(u_\varepsilon)\rightharpoonup
 \eta  \qquad&\text{in } L^2(0,T;H)\,,\label{conv:B_ep}\\
 \mu_\varepsilon \rightharpoonup \mu 
 \qquad&\text{in } L^2(0,T;W)\,,\\
 \xi_\ep \rightharpoonup \xi 
 \qquad&\text{in } L^2(0,T;H)
\end{align}
for some
\[
    u \in H^1(0,T; H)\,, \qquad
    \mu \in L^2(0,T; W)\,, \qquad
    \xi,\eta \in L^2(0,T; H)\,.
\]
We proceed by showing the strong convergence
\begin{equation}\label{strong_u_ep}
    u_\varepsilon \to u 
    \qquad\text{in } C^0([0,T]; H)\,.
\end{equation}
To this end, we show that the sequence
$(u_\ep)_\ep$ is Cauchy in $C^0([0,T]; H)$.
For any arbitrary $\sigma>0$, 
we apply Lemma~\ref{lemma:delta-est} with 
the choice $\delta:=\frac{\sigma}{4C}$, 
where $C>0$ is the constant obtained in
the estimates above. We deduce that there
exists $\bar\ep=\bar\ep_\sigma$ 
and $C_\sigma>0$ such that 
\[
\|(u_{\ep_1}-u_{\ep_2})(t)\|_H^2\leq
\frac{\sigma}{4C}\left(E_{\ep_1}(u_{\ep_1}(t))
+E_{\ep_2}(u_{\ep_2}(t))\right)
+C_\sigma\|(u_{\ep_1}-u_{\ep_2})(t)\|_{V^*}^2
\]
for every $\ep_1,\ep_2\in(0,\bar\ep_\sigma)$,
for every $t\in[0,T]$.
Thanks to \eqref{strong_star}, there 
exists $\tilde\ep_\sigma\in(0,\bar\ep_\sigma)$
such that 
\[
\|u_{\ep_1}-u_{\ep_2}\|^2_{C^0([0,T];V^*)}\leq
\frac{\sigma}{2C_\sigma} \quad\forall\,\ep_1,\ep_2
\in(0,\tilde\ep_\sigma)\,.
\]
Hence, taking the supremum in time
and using the estimates above we infer that 
\begin{align*}
  &\|u_{\ep_1}-u_{\ep_2}\|_{C^0([0,T];H)}^2\\
  &\leq
  \frac{\sigma}{4C}\left(
  \|E_{\ep_1}(u_{\ep_1})\|_{L^\infty(0,T)}+ \|E_{\ep_2}(u_{\ep_2})\|_{L^\infty(0,T)}\right)
  +C_\sigma\|u_{\ep_1}-u_{\ep_2}\|^2_{C^0([0,T];V^*)}\\
  &\leq\frac\sigma{4C}(C+C) + 
  C_\sigma\frac\sigma{2C_\sigma} = \sigma
\end{align*}
for every $\ep_1,\ep_2\in(0,\tilde\ep_\sigma)$.
Since $\sigma>0$ is arbitrary, we obtain the strong
convergence \eqref{strong_u_ep}.

Now, from \eqref{strong_u_ep} 
and the Lipschitz continuity of $\Pi$, it follows that
\[
  \Pi(u_\ep) \to \Pi(u) 
  \qquad\text{in } C^0([0,T]; H)\,,
\]
while the strong-weak closure of $\gamma$
readily ensures that $\xi_\ep\in\gamma(u_\ep)$
almost everywhere in $Q$.

To conclude the proof of the theorem, it remains to
prove additional spatial regularity for $u$ and to
provide an identification of $\eta$. 
First of all, note that since $(u_\ep)_\ep$
is bounded in $L^\infty(0,T; V_\ep)$, 
by the Ponce criterion \cite[Theorem~1.2]{ponce04}
we have that $u\in L^\infty(0,T; V)$.

Let us identify now the term $\eta$.
We first observe that by Lemma \ref{lemma:properties} there holds $DE_\ep=B_\ep$ as operators on $V_\ep$. Thus, by Lemma \ref{lemma:other-prop}, and by the continuous inclusion of $V$ into $V_\ep$, we deduce
\[
E_\varepsilon(z_1)+ \langle B_\varepsilon(z_1),z_2-z_1\rangle_{V_\ep^*,V_\ep}\leq E_\varepsilon (z_2) \quad\forall\,
z_1,z_2\in V\,.
\]
Hence, for all $z\in L^2(0,T; V)$ we deduce that
\begin{equation}
    \label{eq:add20}
    \int_0^TE_\varepsilon(u_\varepsilon(t,\cdot))\, \xd t
    +\int_0^T\int_\Omega 
    \bB_\varepsilon(u_\varepsilon(t,x))
    (z(t,x)-u_\varepsilon(t,x))\, \xd x \, \xd t\leq \int_0^T E_\varepsilon(z(t,\cdot))\, \xd t.
\end{equation}
Owing to Lemma \ref{lemma:other-prop},
and to the dominated convergence theorem, we have
\[
\int_0^T E_\varepsilon(z(t,\cdot))\, \xd t\to
\frac12\int_0^T\int_\Omega|\nabla z(x,t)|^2\xd x\,\xd t\,.
\]
On the one hand, \eqref{conv:B_ep} and \eqref{strong_u_ep} yield
\[
\int_0^T\int_\Omega 
B_\varepsilon(u_\varepsilon(t,x))
(z(t,x)-u_\varepsilon(t,x))\, \xd x\, \xd t \to 
\int_0^T\int_\Omega\eta(t,x)(z(t,x)-u(t,x))\, \xd x\, \xd t.
\] 
On the other hand, by 
the Gamma-convergence result in
Lemma \ref{lemma:other-prop} and by Fatou's Lemma, 
\[
\liminf_{\ep\to 0}\int_0^T E_\varepsilon(u_\varepsilon(t,\cdot))\,\xd t\geq  
\frac12\int_Q|\nabla u(t,x)|^2\,\xd x\,\xd t.
\]
Letting $\varepsilon\to0$
in \eqref{eq:add20} and 
recalling that $u\in L^\infty(0,T; V)$, 
we obtain the inequality 
\begin{equation}
\label{eq:in-fin1}
\frac12\int_Q |\nabla u(t,x)|^2\,\xd x\, \xd t
+\int_Q \eta(t,x) (z(t,x)-u(t,x))\, \xd x\, \xd t
\leq \frac12\int_Q|\nabla z(t,x)|^2\,\xd x\, \xd t
\end{equation}
for every $z\in L^2(0,T; V)$, which in turn implies that
$-\Delta u=\eta \in L^2(0,T;H)$.
Since $u\in L^\infty(0,T; V)$ and 
$\Delta u \in L^2(0,T; H)$ in the sense
of distributions for example, 
by \cite[Thm.~2.27]{brezzi-gilardi} 
the normal derivative 
$\partial_{\bf n} u 
\in L^2(0,T; H^{-1/2}(\partial \Omega))$ is
well defined.
We infer that, for almost every $t\in(0,T)$
and for every $\varphi\in V$,
\[
  \int_\Omega
  \nabla u(t,x)\cdot\nabla\varphi(x)\,\xd x
  =\int_\Omega\eta(t,x)\varphi(x)\,\xd x\,,
\]
from which it follows that 
\[
  -\int_\Omega\Delta u(t,x)
  \varphi(x)\,\xd x
  +\langle\partial_{\bf n}u(t,\cdot), \varphi_{|\partial\Omega}
  \rangle_{H^{-1/2}(\partial\Omega),
  H^{1/2}(\partial\Omega)}
  =\int_\Omega\eta(t,x)\varphi(x)\,.
\]
As $-\Delta u = \eta$ in $L^2(0,T; H)$, 
we infer that 
\[
  \langle\partial_{\bf n}u(t,\cdot), \varphi_0
  \rangle_{H^{-1/2}(\partial\Omega),
  H^{1/2}(\partial\Omega)}=0
  \quad\forall\varphi_0\in H^{1/2}(\Omega)\,,
\]
hence $\partial_{\bf n}u=0$ 
almost everywhere in $\Sigma$.
Now, since we have that 
$\Delta u\in L^2(0,T; H)$
and $\partial_{\bf n}u =0\in 
L^2(0,T; H^{1/2}(\partial\Omega))$,
by the elliptic 
regularity result \cite[Thm.~3.2]{brezzi-gilardi}
we infer that $u\in L^2(0,T; W)$.
Eventually, letting $\ep\searrow0$
in the equations \eqref{eq1_nloc}--\eqref{eq2_nloc}
we obtain 
\[
  \partial_t u - \Delta\mu = 0 \quad\text{in } L^2(0,T; H)
\]
and
\[
  \mu=\tau\partial_t u
  -\Delta u + \xi + \Pi(u)-g \quad\text{in } L^2(0,T; H)\,.
\]
This implies that $u$ is a solution to the local Cahn-Hilliard equation according to 
conditions \eqref{u_loc}--\eqref{init_loc},
in the viscous case $\tau>0$.
This concludes the proof of Theorem \ref{th:conv}
in the case $\tau>0$.

\subsection{The case $\tau=0$}
We consider here the case $\tau=0$, 
so that $\tau_\ep\to0$.

We perform the first estimate as in the previous 
section: we test \eqref{eq1_nloc} by $\mu_\ep$, 
\eqref{eq2_nloc} by $\partial_t u_\ep$, take the difference, and integrate on $Q_t$: 
we obtain
\begin{align*}
&\int_{Q_t}|\nabla\mu_\varepsilon(s,x)|^2\,\xd x \,\xd s
+\tau_\ep \int_{Q_t}
|\partial_t u_\ep(s,x)|^2\,\xd x \,\xd s
+ E_{\ep}(u_\varepsilon(t,\cdot)) + 
\int_\Omega (\hat\gamma+\hat\Pi)
(u_\varepsilon(t,x))\,\xd x \\
&= E_{\ep}(u_{0,\ep})+
\int_\Omega 
(\hat\gamma+\hat\Pi)(u_{0,\ep}(x))\,\xd x
+\int_{Q_t}g_\ep(s,x)\partial_t u_\ep(s,x)\,\xd x\,\xd s\,.
\end{align*}
Using now the additional assumption \eqref{ip1'_conv}
in the case $\tau=0$, we can integrate by parts
with respect to time 
in the last term on the right-hand side
and use the Young inequality as
\begin{align*}
    &\int_{Q_t}g_\ep(s,x)
    \partial_t u_\ep(s,x)\,\xd x\,\xd s\\
    &=-\int_{Q_t}\partial_tg_\ep(s,x)
    u_\ep(s,x)\,\xd x\,\xd s
    +\int_{\Omega}g_\ep(t,x) u_\ep(t,x)\,\xd x
    -\int_{\Omega}g_\ep(0,x) u_{0,\ep}(x)\,\xd x\\
    &\leq\frac12\|g_\ep\|_{H^1(0,T; H)}^2
    +\frac12\int_{Q_t}|u_\ep(s,x)|^2\,\xd x\,\xd s
    +\sigma\int_\Omega|u_\ep(t,x)|^2\,\xd x
    +\frac1{4\sigma}\|g_\ep(t,\cdot)\|_H^2\\
    &\qquad+\frac12\|u_{0,\ep}\|_H^2 + 
    \frac12\|g_\ep(0,\cdot)\|_H^2
\end{align*}
for every $\sigma>0$.
Moreover, note that by the generalized 
Poincar\'e inequality contained in
\cite[Theorem~1.1]{ponce04}, 
there exist constants $C>0$ and 
$\bar\ep\in(0,\ep_0)$, independent of $\ep$
and of $t$, such that
\[
  \int_\Omega|u_\ep(t,x)-(u_\ep(t,\cdot))_\Omega|^2\,\xd x
  \leq C E_\ep(u_\ep(t,\cdot))
  \quad\forall\,\ep\in(0,\bar\ep)\,.
\]
Since $(u_\ep)_\Omega=(u_{0,\ep})_\Omega$, rearranging the
terms and choosing $\sigma>0$ sufficiently small
(independently of $\ep$), we infer that 
\begin{align*}
&\int_{Q_t}|\nabla\mu_\varepsilon(s,x)|^2\,\xd x \,\xd s
+\tau_\ep \int_{Q_t}
|\partial_t u_\ep(s,x)|^2\,\xd x \,\xd s
+E_\ep(u_\ep(t,\cdot))
+ \|u_\varepsilon(t,\cdot)\|_{H}^2 \\
&\leq C\left(E_{\ep}(u_{0,\ep})+
\|u_{0,\ep}\|_H^2+
\int_\Omega 
(\hat\gamma+\hat\Pi)(u_{0,\ep}(x))\,\xd x
+\|g_\ep\|^2_{H^1(0,T; H)}
\right)+\int_{Q_t}|u_\ep(s,x)|^2\,\xd x\,\xd s
\end{align*}
for a certain $C>0$ independent of $\ep$.
Recalling then the assumptions 
\eqref{ip1_conv}--\eqref{ip1'_conv}, 
the Gronwall lemma yields
\[
  \|\nabla \mu_\ep\|_{L^2(0,T; H)} + 
  \|u_\ep\|_{L^\infty(0,T; V_\ep)}
  +\tau_\ep^{1/2}\|\partial_t u_\ep\|_{L^2(0,T; H)}
  \leq C\,,
\]
hence also, by comparison in \eqref{eq1_nloc},
\[
    \|\partial_t u_\ep\|_{L^2(0,T; V^*)}\leq C\,.
\]

At this point, we proceed exactly as in
the previous Section~\ref{ss:tau_pos}, and 
infer that 
\[
\|\xi_\ep\|_{L^2(0,T; L^1(\Omega))}\leq C\,,
\]
which implies, by comparison in \eqref{eq2_nloc},
that 
\[
\|(\mu_\ep)_\Omega\|_{L^2(0,T)}\leq C\,.
\]
We deduce then 
\[
  \|\mu_\ep\|_{L^2(0,T; V)}\leq C\,,
\]
and again, by comparison in \eqref{eq2_nloc}
and by monotonicity of $\gamma$, that
\[
  \|B_\ep(u_\ep)\|_{L^2(0,T; H)} 
  + \|\xi_\ep\|_{L^2(0,T; H)}\leq C\,.
\]

The Aubin-Lions theorems ensure then that, 
up to not relabeled subsequence, as $\ep\searrow0$,
\begin{align}
 \label{strong_star_0}
 u_{\ep}\to u \qquad&\text{in }  C^0([0,T]; V^*)\,,\\
 u_\varepsilon
 \wstarto u \qquad&\text{in } H^1(0,T;V^*)\cap
 L^\infty(0,T; H)\,, \\
 \tau_\ep u_\ep \to 0 
 \qquad&\text{in } H^1(0,T;H)\,,\\
 B_\varepsilon(u_\varepsilon)\rightharpoonup
 \eta  \qquad&\text{in } L^2(0,T;H)\,,\label{conv:B_ep_0}\\
 \mu_\varepsilon \rightharpoonup \mu 
 \qquad&\text{in } L^2(0,T;V)\,,\\
 \xi_\ep \rightharpoonup \xi 
 \qquad&\text{in } L^2(0,T;H)
\end{align}
for some
\[
    u \in H^1(0,T; V^*)\cap L^\infty(0,T; H)\,, \qquad
    \mu \in L^2(0,T; V)\,, \qquad
    \xi,\eta \in L^2(0,T; H)\,.
\]
Arguing as in the previous Section~\ref{ss:tau_pos}
thanks to the Lemma~\ref{lemma:delta-est}, 
the convergence \eqref{strong_star_0}
and the boundedness of $(E_\ep(u_\ep))_\ep$
in $L^\infty(0,T)$ imply the
strong convergence
\[
    u_\varepsilon \to u 
    \qquad\text{in } C^0([0,T]; H)\,.
\]
Hence, by
the Lipschitz continuity of $\Pi$ we have
\[
  \Pi(u_\ep) \to \Pi(u) 
  \qquad\text{in } C^0([0,T]; H)\,,
\]
while the strong-weak closure of $\gamma$
yields $\xi_\ep\in\gamma(u_\ep)$
almost everywhere in $Q$.
Moreover, still arguing as in the previous 
section we obtain that $u\in L^\infty(0,T; V)$,
$\eta=-\Delta u$, and 
$u\in L^2(0,T; W)$ by elliptic regularity.

Passing to the weak limit in 
\eqref{eq1_nloc}--\eqref{eq2_nloc}
we obtain then
\[
  \partial_t u - \Delta\mu = 0 
  \quad\text{in } L^2(0,T; V^*)
\]
and
\[
  \mu=
  -\Delta u + \xi + \Pi(u)-g 
  \quad\text{in } L^2(0,T; H)\,.
\]
This concludes the proof of Theorem \ref{th:conv}
also in the case $\tau=0$.

%%%%%%%%%%%%%%%%%%%%%%%%%%%%%%%%%%%%%%%%%%%%%%%%%%%%%%%%%%%%%%%%%%%%%%%%%%%%%%%
\section*{Acknowledgements}
%%%%%%%%%%%%%%%%%%%%%%%%%%%%%%%%%%%%%%%%%%%%%%%%%%%%%%%%%%%%%%%%%%%%%%%%%%%%%%%

E.D and L.T. have been supported by the Austrian Science Fund (FWF) project F 65. E.D. has been funded by the Austrian Science Fund (FWF) project V 662 N32. The research of E.D. has been additionally supported from the Austrian Science Fund (FWF) through the grant I 4052 N32, and from BMBWF through the OeAD-WTZ project CZ04/2019.

L.S.~has been funded 
by Vienna Science and Technology Fund (WWTF) through Project MA14-009.

\bibliographystyle{abbrv}
\bibliography{ref}
\end{document}